\documentclass[12pt]{amsart}
\usepackage{amssymb}
\usepackage{amscd}
\usepackage{amsmath}
\theoremstyle{plain}
\newtheorem{Thm}{Theorem}[section]
\newtheorem{Thm*}{Theorem}[section]
\newtheorem{Thm'}[Thm]{"Theorem"}
\newtheorem{Cor}[Thm]{Corollary}

\newtheorem{Prop}[Thm]{Proposition}
\newtheorem{Lem}[Thm]{Lemma}

\theoremstyle{definition}

\newtheorem{Emp}[Thm]{}

\newtheorem{Q}[Thm]{Question}

\numberwithin{equation}{section}
\newcommand{\nc}{\newcommand}
\nc{\lm}{\lambda}

\newcommand{\ov}{\overline}

\newcommand{\B}[1]{\mathbb#1}
\newcommand{\cal}[1]{\mathcal{#1}}
\newcommand{\C}[1]{\cal#1}

\newcommand{\sr}{\operatorname{sr}}

\newcommand{\isom}{\overset {\thicksim}{\to}}

\newcommand{\om}{\omega}

\newcommand{\si}[2]{\text{$\Sigma^{#1}_{#2}$}}
\newcommand{\Si}{\Sigma}

\newcommand{\surj}{\twoheadrightarrow}
\newcommand{\lra}{\longrightarrow}

\newcommand{\hra}{\hookrightarrow}
\newcommand{\wt}{\widetilde}
\newcommand{\Gm}{\Gamma}
\newcommand{\lan}{\langle}

\newcommand{\ran}{\rangle}

\newcommand{\g}{\mathfrak{g}}
\renewcommand{\P}{P}
\newcommand{\T}{\Theta}
\renewcommand{\Q}{\mathbf{Q}}
\newcommand{\I}{I}

\newcommand{\dt}{\delta}
\newcommand{\Dt}{\Delta}

\newcommand{\bs}{\backslash}

\newcommand{\m}{^{\times}}

\newcommand{\op}{^{\operatorname{op}}}

\newcommand{\al}{\alpha}

\newcommand{\la}{\lambda}

\newcommand{\form}[1]{(\ref{Eq:#1})}

\newcommand{\rl}[1]{Lemma \ref{L:#1}}

\newcommand{\rp}[1]{Proposition \ref{P:#1}}

\newcommand{\re}[1]{\ref{E:#1}}
\newcommand{\rco}[1]{Corollary \ref{C:#1}}

\newcommand{\rt}[1] {Theorem \ref{T:#1}}

\newcommand{\sm}{\smallsetminus}

\newcommand{\GU}[1]{\bold{GU_{#1}}}

\newcommand{\SL}[1]{\mathbf{SL_{#1}}}
\newcommand{\SU}[1]{\mathbf{SU_{#1}}}

\newcommand{\sgn}{\operatorname{sgn}}

\newcommand{\pr}{\operatorname{pr}}

\newcommand{\Ker}{\operatorname{Ker}}

\newcommand{\val}{\operatorname{val}}

\newcommand{\im}{\operatorname{Im}}

\newcommand{\R}{\operatorname{R}}

\newcommand{\Ad}{\operatorname{Ad}}

\newcommand{\Gal}{\operatorname{Gal}}
\newcommand{\Tr}{\operatorname{Tr}}
\newcommand{\Av}{\operatorname{Av}}

\newcommand{\Supp}{\operatorname{Supp}}

\newcommand{\Par}{\operatorname{Par}}

\newcommand{\nr}{\operatorname{nr}}
\newcommand{\aff}{\operatorname{aff}}

\newcommand{\Lie}{\operatorname{Lie}}
\newcommand{\ad}{\operatorname{ad}}
\newcommand{\cl}{\operatorname{cl}}
\newcommand{\Id}{\operatorname{Id}}

\renewcommand{\sc}{\operatorname{sc}}

\newcommand{\St}{\operatorname{St}}

\newcommand{\End}{\operatorname{End}}

\newcommand{\Ind}{\operatorname{Ind}}

\newcommand{\Irr}{\operatorname{Irr}}

\newcommand{\Stab}{\operatorname{Stab}}
\newcommand{\st}{\operatorname{st}}

\newcommand{\fq}{\B{F}_q}

\newcommand{\wh}{\widehat}

\renewcommand{\si}{\sigma}

\newcommand{\G}{\mathbf{G}}
\newcommand{\CX}{\C{X}}
\newcommand{\CA}{\C{A}}
\newcommand{\GL}{\mathbf{GL}}
\newcommand{\bg}{\boldsymbol{\g}}

\renewcommand{\u}{\mathfrak{u}}
\newcommand{\mm}{\mathfrak{m}}
\renewcommand{\t}{\mathfrak{t}}
\newcommand{\CG}{\bg}
\renewcommand{\H}{\mathbf{H}}
\newcommand{\X}{\mathbf{X}}
\renewcommand{\L}{\mathbf{L}}
\renewcommand{\b}{\mathbf{B}}
\renewcommand{\c}{\mathbf{c}}
\newcommand{\U}{\mathbf{U}}
\newcommand{\M}{\mathbf{M}}
\newcommand{\Z}{\mathbf{Z}}

\renewcommand{\S}{\mathbf{S}}
\newcommand{\TT}{\mathbf{T}}

\begin{document}

\title[On the depth $r$ Bernstein projector]
{On the depth $r$ Bernstein projector}

\author{Roman Bezrukavnikov}
\address{Department of Mathematics\\
Massachusetts Institute of Technology\\
77 Massachusetts Avenue\\
Cambridge, MA 02139, USA} \email{bezrukav@math.mit.edu}

\author{David Kazhdan}
\address{Institute of Mathematics\\
The Hebrew University of Jerusalem\\
Givat-Ram, Jerusalem,  91904\\
Israel} \email{kazhdan@math.huji.ac.il}

\author{Yakov Varshavsky}
\address{Institute of Mathematics\\
The Hebrew University of Jerusalem\\
Givat-Ram, Jerusalem,  91904\\
Israel} \email{vyakov@math.huji.ac.il}




\thanks{The project have received funding from ERC under grant agreement No 669655. This research was partially supported by the BSF grant 2012365}
\begin{abstract}
In this paper we prove an explicit formula for the Bernstein projector to representations of depth $\leq r$.
As a consequence, we show that the depth zero Bernstein projector is supported on topologically unipotent elements
and it is equal to the restriction of the character of the Steinberg representation. As another application,
we deduce that the depth $r$ Bernstein projector is stable. Moreover, for integral depths our proof is purely local.
\end{abstract}
\maketitle

\centerline{\em To Iosif Bernstein with gratitude and best wishes on his birthday}
\tableofcontents

\section*{Introduction}

\noindent Let $G$ be a reductive $p$-adic group. Recall that the {\em Bernstein center} $Z_G$
of $G$ is a commutative ring which plays a role in representation theory of $G$ similar
to the role played by the center of the group ring  in representation theory of a finite
group.

Elements of $Z_G$ can be thought of as invariant distributions on $G$. While the Bernstein center is an important tool
in the structure theory of representations of $G$, known explicit formulas for its elements are rather rare. In this paper we provide
explicit descriptions for some natural elements in $Z_G$.

Recall also that $Z_G$ admits a natural injective homomorphism  into the ring of functions on the set $\Irr(G)$ of irreducible
smooth representations.

Fix a number $r\geq 0$ and consider the function $f_r$ on $\Irr(G)$ such that
$f_r(V)=1$ if the depth of $V$ is $\leq r$,  and $f_r(V)=0$ otherwise.
The main results of this paper describe the element
$E_r\in Z_G$ for which the corresponding function on $\Irr(G)$ equals $f_r$.
We call $E_r$ {\em the depth $r$ projector}.

The  first result (available only for $r=0$) is the equality between $E_0$ and the restriction of
the character of the Steinberg representation to the locus of topologically unipotent elements
of $G$. This can be thought of as a $p$-adic group analogue of the well-known fact that the
character of the Steinberg representation of a finite Chevalley
group restricted to the set of unipotent elements is proportional to the delta function
of the unit element.

Let $\g^*$ be the linear dual of the Lie algebra $\g$ of $G$. Our second result describes
$E_r$ in terms of the Fourier transform of the characteristic function of a certain  subset of $\g^*$.
This formula fits naturally into the standard
analogy between harmonic analysis on the group $G$ and on its Lie algebra $\g$ (notice that
under this analogy elements of $Z_G$ correspond to invariant distributions on
$\g$ whose Fourier transform is locally constant).

As a corollary of our description, we show that $E_r$ is a {\em stable distribution}.
This property of $E_r$ is suggested by the conjectural theory of $L$-packets and its
relation to endoscopy for invariant distributions. The set $\Irr(G)$ is conjectured to
be partitioned into finite subsets called $L$-packets; among many expected properties
of $L$-packets we mention the following: an element $E\in Z_G$ is a stable distribution
if and only if the corresponding function on $\Irr(G)$ is constant on $L$-packets.
It is also expected that the set of irreducible representations of a given depth is a union
of $L$-packets, thus the above conjectures imply that $E_r$ is a stable distribution;
we prove this fact unconditionally.

This result also provides evidence for another conjecture which has the advantage
of being a self-contained formal statement. The so called stable center conjecture
asserts that the subspace of stable distributions in $Z_G$ is a subring. It follows
from our results that  the space of stable distributions in $Z_G$ does contain a subring:
the linear span of the projectors $E_r$ ($r\geq 0$).

\medskip

This work is an outgrowth of a project described in \cite{BKV} whose goal is to construct
elements in $Z_G$ and more general invariant distributions of interest using $l$-adic
sheaves on loop groups. Such a construction for $E_0$ (for split groups in positive
characteristic) was presented in \cite{BKV}.

Though $l$-adic sheaves are not used in the present paper, our main technical result,
\rt{explicit}, was suggested by \cite{BKV}. Namely, the $l$-adic sheaf counterpart of $E_0$
can be constructed by taking derived invariants of the affine Weyl group $W^{\aff}$ acting on the loop group version
of the Springer sheaf. Moreover, using a standard resolution for the trivial representation
of $W^{\aff}$ whose terms are indexed by standard parabolic subgroups therein, we get
an explicit resolution for this sheaf. This leads to the formula for the corresponding
function appearing in \rt{explicit}.

Our method was motivated by a work of Meyer--Solleveld \cite{MS}, who generalized a work of
Schneider--Stuhler \cite{SS}.

\medskip

{\bf Acknowledgements.} We thank Akshay Venkatesh whose question motivated us to rewrite a geometric formula from
\cite{BKV} in elementary terms. We also thank Gopal Prasad for stimulating conversations
and  Ju-Lee Kim and Allen Moy for useful discussions. We thank Dennis Gaitsgory and the referee for their comments
and suggestions.

\section{Statement of results}
\begin{Emp} \label{E:not0}
\noindent{\bf Notation.}
(a) Let $F$ be a local non-archimedean field of residual characteristic $p$, $\ov{F}$ an algebraic closure of $F$, $F^{\nr}\subseteq\ov{F}$  the maximal unramified extension of $F$ inside $\ov{F}$, and $\val_{F}$ the valuation on $\ov{F}$ such that $\val_F(F\m)=\B{Z}$.

(b) Let $\G$ be a connected reductive group over $F$, $G:=\G(F)$, and
 $\CX=\CX(\G)$ the reduced Bruhat--Tits building of $\G$, viewed as a metric space, and equipped with extra structure
(see \re{bt}). To every pair  $(x,r)\in \CX\times\B{R}_{\geq 0}$, Moy--Prasad \cite{MP1,MP2} associate an open-compact subgroup
$G_{x,r^+}\subseteq G$ (see \re{mpsplit} and \re{mp}).
\end{Emp}

\begin{Emp} \label{E:depth}
\noindent{\bf Depth of a representation.}
(a) Let $R(G)$ be the category of smooth complex representations of $G$, and let $\Irr(G)$ be the set of equivalence classes
of irreducible objects of $R(G)$.  To each $V\in R(G)$, Moy--Prasad associate a depth $r\in \B{Q}_{\geq 0}$, which is defined to be
the smallest $r\in\B{R}_{\geq 0}$ such that $V^{G_{x,r^+}}\neq 0$ for some $x\in\CX$. Actually, for our purposes slightly weaker results of DeBacker (\cite{DB}) are sufficient.

(b) For every $r\in \B{Q}_{\geq 0}$, we denote by $\Irr(G)_{\leq r}$ (resp. $\Irr(G)_{>r}$) the set of $V\in\Irr(G)$ of depth $\leq r$ (resp. $>r$), and denote by   $R(G)_{\leq r}$ (resp. $R(G)_{>r}$) the full subcategory of $R(G)$ consisting of
representations $V$ all of whose irreducible subquotients belong to $\Irr(G)_{\leq r}$ (resp. $\Irr(G)_{>r}$).

(c) It follows from a combination of results of Bernstein \cite{Be} and Moy--Prasad
(or DeBacker) that for every $r\in\B{Q}_{\geq 0}$ and $V\in R(G)$ there exists a unique direct sum decomposition $V=V_{\leq r}\oplus V_{>r}$ such that $V_{\leq r}\in R(G)_{\leq r}$ and $V_{>r}\in R(G)_{>r}$. We provide an alternative proof of this fact in \re{idemp}.
\end{Emp}

\begin{Emp} \label{E:cent}
{\bf The Bernstein center.} (a) Let $Z_G$ be the algebra of endomorphisms of the identity functor $\End \Id_{R(G)}$. It is  called the Bernstein center of $G$.
In particular, for every $z\in Z_G$ and $V\in R(G)$, we are given
an endomorphism $z|_V\in\End V$.

(b) Let $\C{H}(G)$ be the algebra of smooth measures with compact support on $G$. Then $\C{H}(G)$ is a smooth representation of $G$
with respect to the left action, and the map $z\mapsto z|_{\C{H}(G)}$ identifies $Z_G$
with the algebra $\End_{\C{H}(G)\otimes\C{H}(G)^{\op}}\C{H}(G)$ of endomorphisms of $\C{H}(G)$, commuting with the left and the right convolution.

(c) For every $V\in R(G)$ and $v\in V$ the map $h\mapsto h(v)$ defines a $G$-equivariant map $\C{H}(G)\to V$. Therefore for every $h\in\C{H}(G)$ and $z\in Z_G$ we have  $z(h(v))=(z(h))(v)$.

(d) For every $z\in Z_G$ there exists a unique invariant distribution $E_z\in D^{G}(G)$ on $G$ such that $z(h)=E_z\ast h$ for every $h\in\C{H}(G)$, where $\ast$ denotes the convolution. Moreover, the map $z\mapsto E_z$ identifies $Z_G$ with the space of all invariant distributions $E\in D^{G}(G)$  such that the distribution $E\ast h$
has a compact support for every $h\in\C{H}(G)$.
\end{Emp}

\begin{Emp} \label{E:proj}
{\bf The Bernstein projector.}
(a) By \re{depth}(c),  there exists an idempotent $\Pi_r\in Z_G$ such that for every object $V\in R(G)$ the endomorphism
$\Pi_r|_V$ is the projection $V\surj V_{\leq r}\hra V$.  We call $\Pi_r$ the {\em depth $r$ Bernstein projector}.

(b) Let $E_r$ be the invariant distribution on $G$, corresponding to $\Pi_r$ (see \re{cent}(d)).
\end{Emp}

A particular case of the stable center conjecture (see \cite{BKV}) asserts that the distribution $E_r$ is stable. The goal of this work is to give an explicit formula for the Bernstein projector $\Pi_r$, and to use this description to show the stability of $E_r$.

From now on we fix $m\in\B{N}$ and $r\in\frac{1}{m}\B{Z}_{\geq 0}$.

\begin{Emp} \label{E:tm}
{\bf Notation.} (a) Denote by $[\CX]$ the set of open polysimplices of $\CX$ (see \re{poly}(a)), and by  $[\CX_m]$ the set of open polysimplices of $\CX$, obtained by ``subdividing each polysimplex $\si\in [\CX]$ into $m^{\dim\si}$ smaller polysimplices" (see \re{ref}(c)).

(b) For every $\si\in [\CX_m]$, we choose $x\in \si$ and define $G_{\si,r^+}:=G_{x,r^+}$. Since $r\in\frac{1}{m}\B{Z}$, the subgroup $G_{\si,r^+}$ does not depend on the choice of  $x$ (see \rl{mpfil}).

(c) To every finite subset $\Si\subseteq [\CX_m]$ we  associate an element
\[
E_r^{\Si}=\sum_{\si\in\Sigma}(-1)^{\dim\si}\dt_{G_{\si,r^+}}\in\C{H}(G),
\]
 where $\dt_{G_{\si,r^+}}$ is the Haar measure on $G_{\si,r^+}$ with total measure one.

 (d) We denote by $\T_m$ the set of non-empty finite convex subcomplexes $\Si\subseteq [\CX_m]$ (see \re{notcomb}),
 and set $\T:=\T_1$. Note that $\T_m$ is an inductive system with respect to inclusions.
\end{Emp}

The following result provides an explicit formula for the projector $\Pi_r$.

\begin{Thm} \label{T:explicit}
For every $V\in R(G)$ and $v\in V$, the inductive system $\{E_r^{\Si}(v)\}_{\Si\in\T_m}$ stabilizes, and $\Pi_r(v)$ equals the limit value of
$E_r^{\Si}(v)$, that is, $\Pi_r(v)=\lim_{\Si\in\T_m}E_r^{\Si}(v)$.
\end{Thm}

\begin{Emp} \label{E:strategy}
{\bf Strategy of the proof.} Analyzing combinatorics of the Bruhat--Tits building, we show that for every $x\in \CX$ and $s\in\B{R}_{\geq 0}$ the inductive system  $\{E_r^{\Si}\ast \dt_{G_{x,s^+}}\}_{\Si\in\T_m}$ stabilizes. This implies that the inductive system  $\{E_r^{\Si}\ast h\}_{\Si\in\T_m}$ stabilizes for all  $h\in\C{H}(G)$, and that there exists a unique element of the Bernstein center $z\in Z_G$ such that $z(h)=\lim_{\Si\in\T_m}E_r^{\Si}\ast h$.

Next, using \re{cent}(c), we show that for every $V\in R(G)$ and $v\in V$, the inductive system  $\{E_r^{\Si}(v)\}_{\Si\in\T_m}$ stabilizes, and
$z(v)=\lim_{\Si\in\T_m}E_r^{\Si}(v)$. In particular, $z|_{V}=0$ for every $V\in \Irr(G)_{>r}$.

It remains to show that $z=\Pi_r$. By a theorem of Bernstein, we have to check that  $z|_{V}=\Id$ for every $V\in \Irr(G)_{\leq r}$.
Using \re{cent}(c) again, it remains to show that $z(\dt_{G_{x,r^+}})=\dt_{G_{x,r^+}}$ for every $x\in \CX$.
To prove this, we show a stronger assertion that $E_r^{\Si}\ast\dt_{G_{x,r^+}}=\dt_{G_{x,r^+}}$ for all $\Si\in\T_m$ such that $x\in\Si$.
\end{Emp}

\begin{Emp}
{\bf Remark.} Our argument also provides an alternative proof of the decomposition $V=V_{\leq r}\oplus V_{>r}$ from \re{depth}(b),
hence an alternative proof of the existence of the projector $\Pi_r$ (see \re{idemp}).
\end{Emp}

Consider the open and closed subset $G_{r^+}:=\cup_{x\in\CX}G_{x,r^+}\subseteq G$ (see \rl{clos} or \cite[Cor 3.7.21]{ADB}). Notice that $G_{0^+}$ is usually called the set of topologically unipotent elements.  \rt{explicit} has the following consequence.

\begin{Cor} \label{C:explicit}
(a) We have the equality $E_r=\lim_{\Si\in\T_m}E^{\Si}_r$. In other words, for every $f\in C_c^{\infty}(G)$ the inductive system  $\{E_r^{\Si}(f)\}_{\Si\in\T_m}$ stabilizes, and $E_r(f)=\lim_{\Si\in\T_m}E^{\Si}_r(f)$.

(b) The invariant distribution $E_r$ is supported on $G_{r^+}$.
\end{Cor}

As a further consequence, we get the following variant of the character formula of Meyer--Solleveld \cite{MS}.

\begin{Cor} \label{C:formchar}
For every  admissible $V\in R(G)_{\leq r}$ and every $h\in\C{H}(G)$ we have
\[
\Tr(h,V)=\lim_{\Si\in\T_m} \left[\sum_{\si\in\Si}(-1)^{\dim\si} \Tr(\dt_{G_{\si,r^+}}\ast h\ast\dt_{G_{\si,r^+}},V^{G_{\si,r^+}})\right].
\]
\end{Cor}

\begin{Emp} \label{E:Av}
{\bf Averaging.} We fix $n\in\B{N}$ and an Iwahori subgroup $\I$ of $G$.

(a) Denote by $\Par$ the set of standard parahoric subgroups $\P\supseteq\I$.
Each $\P\in\Par$ corresponds to a polysimplex $\si_{\P}\in[\CX]$. We $\P^+_n:=G_{\si_{\P},n^+}$ and $\P^+:=\P^+_0$.

(b) For a finite subset $Y\subseteq G/\P$ and an $\Ad P$-invariant distribution $E\in D^{\P}(G)$ we denote by
$\Av_Y(E)\in D(G)$ the distribution $\sum_{g\in Y}(\Ad g)_*(E)$.

(c) For every subset $\Si\subseteq[\CX]$ and $\P\in\Par$ we denote by $Y_{\P}^{\Si}\subseteq G/\P$ the set of all $g\in G/\P$ such that
$g(\si_{\P})\in\Si$.

(d) Since $D(G)$ is the linear dual of $C_c^{\infty}(G)$, the center $Z_G$ acts on $D(G)$ by the formula $z(E)(f)=E(z(f))$ for all $z\in Z_G$, $E\in D(G)$ and $f\in C_c^{\infty}(G)$.
\end{Emp}

Note that for every $n\in\B{N}$ and $E\in D^G(G)$ the distribution $E\ast\dt_{\P^+_n}\in D(G)$ is $\Ad\P$-invariant, thus we can form
$\Av_Y(E\ast\dt_{\P^+_n})\in D(G)$ (see \re{Av}(b)).

\rt{explicit} has the following consequence.

\begin{Cor} \label{C:Av}
For every $E\in D^G(G)$ and $n\in\B{N}$, we have the equality
\begin{equation} \label{Eq:aver}
\Pi_n(E)=\lim_{\Si\in\T}\left[\sum_{\P\in\Par}(-1)^{\dim \si_{\P}}\Av_{Y_{\P}^{\Si}}(E\ast\dt_{\P^+_n})\right].
\end{equation}
In particular, the support of $\Pi_n(E)$ is contained in
$\cup_{\P\in\Par}\Ad G(\Supp (E\ast\dt_{\P^+_n}))$.
\end{Cor}

\begin{Emp}
{\bf Notation.} Let $\mu^{\I^+}$ be the Haar measure on $G$ normalized by  $\int_{\I^+}\mu^{\I^+}=1$.
\end{Emp}

Since the invariant distribution $E_0$ is supported on $G_{0^+}$ (by \rco{explicit}(b)), the following result describes  $E_0$ in terms of the character  $\chi_{\St_G}$ of the Steinberg representation $\St_G$ of $G$.

\begin{Thm} \label{T:steinberg}
We have the equality $E_0|_{G_{0^+}}=(\chi_{\St_G}|_{G_{0^+}})\cdot\mu^{\I^+}$, that is, $E_0(f)$ equals $\chi_{\St_G}(f\mu^{\I^+})$
for every $f\in C_c^{\infty}(G_{0^+})$.
\end{Thm}

To prove this result, we compare the explicit formula for $E_0$ given in \rco{explicit} with a corresponding formula of Meyer--Solleveld \cite{MS}
for $\chi_{\St_G}$.

\begin{Emp} \label{E:comp}
{\bf Remark.} Though our formula from \rco{formchar} applies in a more general situation than a similar formula of
Meyer--Solleveld \cite{MS}, it is less precise. In particular,  \rco{formchar} does not suffice for the proof of
\rt{steinberg}.
\end{Emp}

Since the character of the Steinberg representation is known to be stable (see \re{stsubs}(b)), we deduce from \rt{steinberg} the following  corollary.

\begin{Cor} \label{C:stable}
The invariant distribution $E_0$ is stable.
\end{Cor}

\begin{Emp} \label{E:mplie}
{\bf The Moy--Prasad filtration for the Lie algebras.}
(a) Let $\g$ be the Lie algebra of $\G$, and let $\g^*$ be the dual vector space.
For every $(x,r)\in\CX\times\B{R}_{\geq 0}$, Moy--Prasad define $\C{O}$-lattices $\g_{x,r^+}\subseteq \g$ and
$\g^*_{x,-r}\subseteq \g^*$ (see \re{mpsplit}(c) and \re{mp}(c)).

(b) As in the group case, for every $\si\in [\CX_m]$  we define $\g_{\si,r^+}:= \g_{x,r^+}$ for $x\in\si$ (use \rl{mpfil}). Also
to every $\Si\in\T_m$ we associate an element
\[
\C{E}_r^{\Si}:=\sum_{\si\in\Sigma}(-1)^{\dim\si}\dt_{\g_{\si,r^+}}\in\C{H}(\g).
\]
Here $\C{H}(\g)$ denotes the space of smooth measures with compact support on $\g$, and $\dt_{\g_{\si,r^+}}$ is the Haar measure on $\g_{\si,r^+}$ with total measure one.

(c) Consider the open-closed subsets
$\g_{r^+}:=\cup_{x\in\CX}\g_{x,r^+}\subseteq \g$ and $\g^*_{-r}:=\cup_{x\in\CX}\g^*_{x,-r}\subseteq \g^*$
(see \rl{clos} or \cite[Cor 3.4.3]{ADB}), and denote by $1_{\g^*_{-r}}$ the characteristic function of $\g^*_{-r}$.
\end{Emp}

\begin{Emp}
{\bf The Fourier transform.}
(a) Let $\C{O}\subseteq F$ be the ring of integers, let $\varpi\in \C{O}$ be a uniformizer, and
let $\psi:F\to \B{C}\m$ be an additive character, trivial on $(\varpi)$ but nontrivial on $\C{O}$.
Then $\psi$ gives rise to a Fourier transform $\C{F}:\C{H}(\g^*)\to C_c^{\infty}(\g)$,
 where $\C{H}(\g^*)$ denotes the space of smooth measures with compact support on $\g^*$. Explicitly,
$\C{F}(h)(a)=\int_{\g^*}\psi(\lan\cdot,a\ran)h$ for every $h\in\C{H}(\g^*)$ and $a\in\g$.

(b) By duality, $\C{F}$ gives rise to an isomorphism $\C{F}:D^G(\g)\isom\wh{C}^G(\g^*)$ between
the space of  invariant distributions on $\g$ and the space of invariant generalized functions on $\g^*$.
Explicitly, $\C{F}(E)(h)=E(\C{F}(h))$ for every $E\in D^G(\g)$ and $h\in\C{H}(\g^*)$.
\end{Emp}

\begin{Emp}
{\bf The Lie algebra analogue of the center.}
(a) We denote by $Z_{\g}\subseteq D^G(\g)$ the subspace of all $E$ such that the distribution $E\ast h$ has compact support
for every $h\in\C{H}(\g)$. Equivalently, $E\in D^G(\g)$ belongs to $Z_{\g}$ if and only if the Fourier transform
$\C{F}(E)\in\wh{C}^G(\g)$ is locally constant.

(b) We set $\C{E}_r:=\C{F}^{-1}(1_{\g^*_{-r}})\in  D^G(\g)$, and call $\C{E}_r$ {\em the Lie algebra analogue of the depth
$r$ projector}. Since $\g^*_{-r}\subseteq \g^*$ is open and closed, the function  $1_{\g^*_{-r}}$ is locally constant,
thus $\C{E}_r\in Z_{\g}$.
\end{Emp}

The following result is the Lie algebra analogue of \rco{explicit}.

\begin{Prop} \label{P:lie}
For every $f\in C_c^{\infty}(\g)$ the inductive system  $\{\C{E}_r^{\Si}(f)\}_{\Si\in\T_m}$ stabilizes, and
$\C{E}_r(f)=\lim_{\Si\in\T_m}\C{E}^{\Si}_r(f)$. In particular, $\C{E}_r$ is supported on $\g_{r^+}$.
\end{Prop}

\begin{Emp}
{\bf An $r$-logarithm.}  By an {\em $r$-logarithm} we mean an $\Ad G$-equivariant homeomorphism
$\C{L}:G_{r^+}\isom \g_{r^+}$, which induces a homeomorphism $\C{L}_{x}:G_{x,r^+}\isom\g_{x,r^+}$ for all $x\in \CX$.
\end{Emp}

\begin{Cor} \label{C:lie}
Let $\C{L}:G_{r^+}\isom\g_{r^+}$ be an $r$-logarithm. Then the pushforward
$\C{L}_!(E_r|_{G_{r^+}})$ equals $\C{E}_r|_{\g_{r^+}}$.
\end{Cor}

By a theorem of Waldspurger, the Fourier transform preserves stability (see \cite{Wa} or \cite{KP});
therefore \rco{lie} easily implies that $E_r$ is stable if $\G$ admits an $r$-logarithm (see \rco{wald}). Furthermore, extending the theory
of quasi-logarithms, introduced in \cite{KV}, we show the following result.

\begin{Thm} \label{T:stable}
Assume that  $p$ is very good for $G$ (see \re{verygood}).  Then the invariant distribution $E_r$ is stable.
\end{Thm}


\begin{Emp} \label{E:char0}
{\bf Remarks.} (a) If $F$ is of characteristic zero, one can show that $E_r$ is stable if $p$ is good (see \re{verygood} and \re{good}). In particular, in this case  $E_r$ is stable if $p>5$.

(b) Notice that since the proof of a theorem of Waldspurger is global, for a general $r$ our proof of the stability of $E_r$ is global. On the other hand, when $r\in\B{N}$, we can deduce the stability of $E_r$ from that of $E_0$ (see \re{remwald}(c)), thus providing a purely local proof in this case.

(c) Allen Moy has informed us that he has independently conjectured \rco{lie}
(for large $r$ and fields of characteristic zero), found a proof for
$\G=\SL{2}$ and discovered its relation to the stability of the Bernstein projectors (see \cite{Mo}).
\end{Emp}

\begin{Emp}
{\bf Plan of the paper.} The paper is organized as follows. In Section 2 we review basic properties of Bruhat--Tits buildings and then present the
construction in the split case. In Section 3 we recall the construction and basic properties of Moy--Prasad filtrations, first in the split case
and then in general. In order to make our presentation more elementary, we do not use N\'eron models.

In Sections 4-5 we prove the  stabilization assertion needed for \rt{explicit}. Then, in Section 6, we complete the proof of \rt{explicit},
deduce Corollaries \ref{C:explicit}, \ref{C:formchar} and \ref{C:Av}, and prove the Lie algebra analogues (\rp{lie} and \rco{lie}).

In Section 7 we compare the projector to depth zero with the character of the Steinberg representation (\rt{steinberg}).
Finally, in Section 8, we show the stability of the projector
(\rco{stable} and \rt{stable}).

In appendices we prove several assertions, stated in the main part of the paper without proofs. Namely, in Appendix A
we provide details on various properties of the Moy--Prasad filtrations, well-known to specialists, formulated in Section 3. In Appendix B
we study congruence subsets, used in Section 8.

Finally, in Appendix C we review the theory of the quasi-logarithms introduced in \cite{KV,KV2} and deduce the existence of $r$-logarithms. This is used
in the proof of \rt{stable}, and has other applications as well.
\end{Emp}

\begin{Emp}
{\bf General case versus split case.} The constructions of Bruhat--Tits buildings and Moy--Prasad filtrations
are much more transparent when $\G$ is split. On the other hand, once Bruhat--Tits buildings and Moy--Prasad filtrations
are constructed and their properties are established, the argument in the general case is identical to the split one.
\end{Emp}




\section{Bruhat--Tits buildings} \label{S:bt}

\noindent In this section we formulate basic properties of Bruhat--Tits buildings (see \cite{BT1,BT2}) and then review the construction
in the case when $\G$ is split.

\begin{Emp} \label{E:bt}
{\bf The Bruhat--Tits building.} Let $\G^{\ad}$ be the adjoint group of $\G$.

(a) For every maximal split torus $\S\subseteq\G$, we denote by $\S_{\G^{\ad}}$ the corresponding maximal split torus of $\G^{\ad}$ and
consider the $\B{R}$-vector space $V_{\G,\S}:=X_*(\S_{\G^{\ad}})\otimes_{\B{Z}}\B{R}$, where $X_*(\cdot)$ denotes the group of cocharacters.
We equip each $V_{\G,\S}$ with a $W(\G,\S)$-invariant inner product, such that for every $g\in G$ the induced map
$\Ad g: V_{\G,\S}\isom  V_{\G,g\S g^{-1}}$ is orthogonal.

(b) We denote by $\C{A}_{\S}=\C{A}_{\G,\S}$ the ``canonical" affine space under $V_{\G,\S}$ (see  \re{apart} below in the split case
and \cite[1.2]{Ti} or \cite[1.9]{La}, in general). We equip each $\C{A}_{\S}$ with a metric induced by the inner product on $V_{\G,\S}$, chosen in (a).
$\CA_{\S}$ is called the {\em apartment corresponding to $\S$}.


(c) The {\em (reduced) Bruhat--Tits building} $\C{X}=\C{X}(\G)$ of $\G$ is a $G$-equivariant metric space $\C{X}=\C{X}(G)$, equipped with a decomposition  $\C{X}=\cup_{\S}\C{A}_{\S}$ into a union of apartments, indexed by maximal split tori, such that each inclusion $\CA_{\S}\hra\CX$ is distance preserving.
\end{Emp}

\begin{Emp} \label{E:rembt}
{\bf Remarks.} (a) $\C{X}(\G)$ depends only on the adjoint group $\G^{\ad}$.

(b) If $\G=\prod_i\G_i$, then $\CX(\G)$ is the product $\prod_i\CX(\G_i)$. In particular, study of $\CX(\G)$ often reduces to the case when $\G$ is simple and adjoint.

(c) If $\G$ is simple, then a metric on $\CX(\G)$ is uniquely defined up to a multiplication by a scalar.
\end{Emp}

%

\begin{Emp} \label{E:root}
{\bf Affine root subgroups.}
Let $\S\subseteq\G$ be a maximal split torus.

(a) For every root $\al\in \Phi(\G,\S)$, we denote by $\u_{\al}\subseteq\g$ the corresponding root subspace. We also denote by
$\U_{\al}\subseteq\G$ the corresponding root subgroup (see \cite[21.9]{Bo2}), and set $U_{\al}:=\U_{\al}(F)$.
By definition, $\U_{\al}$ is a connected unipotent group, whose Lie algebra is $\u_{\al}\oplus\u_{2\al}$.
There exists a canonical isomorphism  $U_{\al}/U_{2\al}\isom\u_{\al}$, hence a canonical surjection $\iota_{\al}:U_{\al}\to\u_{\al}$.

Note that if  $\G$ is split, then both $\u_{2\al}$ and $U_{2\al}$ are trivial, thus $\iota_{\al}$ is an isomorphism.

(b) Let $\C{A}:=\C{A}_{\S}$ be the apartment, corresponding to $\S$. We denote by $\Psi(\C{A})$ the set of affine roots (see \cite[1.6]{Ti}). Each $\psi\in \Psi(\C{A})$ is an affine function of $\C{A}$, whose vector part $\al=\al_{\psi}\in (V_{\G,\S})^*$
belongs to $\Phi(\C{A}):=\Phi(\G,\S)$.

(c) We denote by $U_{\psi}\subseteq U_{\al}$ the affine root subgroup corresponding to $\psi$
(see \cite[1.4]{Ti}), and we set $\u_{\psi}:= \iota_{\al}(U_{\psi})\subseteq \u_{\al}$.
Then $\u_{\psi}\subseteq \u_{\al}$ is an $\C{O}$-submodule (see \re{lattice}(a)).
\end{Emp}

\begin{Emp} \label{E:buildings}
{\bf Properties of buildings.}
The following standard properties of the Bruhat--Tits building $\CX$ will be used later.

(a) Every two points $x,y\in\CX$ belong to an apartment (see \cite[Prop. 13.12]{La}).

(b) For every two apartments $\C{A},\C{A}'\subseteq\C{X}$ there exists a distance preserving isomorphism of affine spaces
$\C{A}\isom \C{A}'$, which is the identity on $\C{A}\cap\C{A}'$ and induces a bijection $\Psi(\CA')\isom \Psi(\CA)$ between the sets of affine roots
(see \cite[Prop 13.6]{La}).

(c) For every two points $x,y\in\CX$ there exists a unique geodesic $[x,y]\subseteq\C{X}$. Moreover, $[x,y]$ is a geodesic in $\CA$ for every apartment $\CA\ni x,y$ (by (b)).
\end{Emp}

\begin{Emp} \label{E:bc}
{\bf Base change.}
For a finite Galois extension $K/F$, we denote by $\G_K$ the base change of $\G$ from $F$ to $K$.
Then the building $\CX(\G_{K})$ is equipped with an action of the Galois group $\Gal(K/F)$, and we have a natural $G$-equivariant embedding $\CX(\G)\hra\CX(\G_{K})^{\Gal(K/F)}$. Moreover, the latter inclusion is an isomorphism
if $K/F$ is unramified. We denote the image of $x\in \CX(\G)$ in $\CX(\G_{K})$ simply by $x$.
\end{Emp}

\begin{Emp} \label{E:poly}
{\bf Polysimplicial decomposition.}
(a) The Bruhat--Tits building $\C{X}$ is equipped with a decomposition into a disjoint union
of (open) polysimplices, that is, products of simplices (see (b) below). Moreover, each apartment $\CA\subseteq\CX$ is a union of polysimplices. We denote by
$[\CX]$ (resp. $[\CA]$) the set of polysimplices in $\CX$ (resp. $\CA$).

(b) More precisely, two points $x,y\in \CA$ belong to a polysimplex if and only if for every $\psi\in\Psi(\CA)$ we have $\psi(x)\geq 0$ if and only if $\psi(y)\geq 0$, while two points $x,y\in \CX$ belong to a polysimplex if and only if they belong to a polysimplex in $\CA$ for some apartment $\C{A}$ containing $x$ and $y$
(see \re{buildings}(a)).

(c) By property \re{buildings}(b), if two points $x,y\in \CX$ belong to a polysimplex, then they belong to a polysimplex in $\CA$ for every apartment $\C{A}$ containing $x$ and $y$.

(d) It follows from \re{buildings}(a) that for every pair of polysimplices $\si,\tau\in[\CX]$ there exists  an apartment $\C{A}\supseteq\si,\tau$.
\end{Emp}

\begin{Emp} \label{E:ref}
{\bf Refined affine roots.} Let $\C{A}\subseteq\C{X}$ be an apartment and $m\in\B{N}$.

(a) For every $\psi\in\Psi(\C{A})$ there exists $n_{\psi}\in\B{Z}_{>0}$ such that
the set of $\psi'\in \Psi(\C{A})$, whose vector part is $\al_{\psi}$, equals $\psi+\frac{1}{n_{\psi}}\B{Z}$ (see \re{lattice}(c)).
In particular, we have $\psi+\B{Z}\subseteq  \Psi(\C{A})$ for every $\psi\in \Psi(\C{A})$.
Note that if $\G$ is split, then  $n_{\psi}=1$ for all $\psi$ (see \re{rig}(a)).

(b) We denote by $\Psi_m(\CA)$ the set of affine functions on $\C{A}$ of the form
$\psi+\frac{k}{m n_{\psi}}$, where $\psi\in\Psi(\C{A})$ and  $k\in\B{Z}$. In particular, $\Psi_m(\CA)\supseteq\Psi(\CA)$, and for every $\psi\in\Psi_m(\CA)$, we have $\psi+\frac{1}{m}\B{Z}\subseteq\Psi_m(\CA)$.

(c) We denote by  $[\CX_m]$ (resp. $[\CA_m]$) the set of polysimplices in $\CX$ (resp. $\CA$), obtained by the
same procedure as in \re{poly}(b), but replacing $\Psi(\CA)$ by $\Psi_m(\CA)$. Alternatively,
polysimplices in $[\CX_m]$ (resp. $[\CA_m]$) are obtained by a subdivision of each polysimplex $\si\in[\CX]$ (resp. $\si\in[\CA]$) into $m^{\dim\si}$ smaller polysimplices.
\end{Emp}

For the convenience of the reader, we now recall the construction of the building $\CX(\G)$ when $\G$ is split. Replacing $\G$ by $\G^{\ad}$ and decomposing $\G$ into a product of simple factors, we can assume that $\G$ is simple and adjoint (see \re{rembt}(a),(b)).






\begin{Emp} \label{E:rig}
{\bf Notation.}
Let $\S\subseteq\G$ be a maximal split torus.

(a) Consider the vector space $V_{\G,\S}:=X_*(\S)\otimes_{\B{Z}}\B{R}$. Define the set $\Psi(\G,\S)$ of {\em affine roots} as the set of affine functions on
$V_{\G,\S}$ of the form $\psi_{\al,k}:=\al+k$ where $\al\in\Phi(\G,\S)$ and $k\in\B{Z}$. Note that
the lattice $X_*(\S)$ acts on $V_{\G,\S}$ by translations, and the set $\Psi(\G,\S)$ of affine roots is $X_*(\S)$-invariant.

(b) The adjoint action of $\S$ on $\g$ defines a decomposition $\g=\u_0\oplus(\oplus_{\al\in\Phi(\G,\S)}\u_{\al})$ into a direct sum of weight spaces,
where $\u_0=\Lie\S$. Also for every $\al\in\Phi(\G,\S)$ we have a canonical isomorphism $\iota_{\al}:U_{\al}\isom\u_{\al}$.

(c) Since $\S$ is a split torus, it has a natural structure $\frak{S}$ over $\C{O}$. By an {\em $\C{O}$-structure} of $(\g,\S)$, we mean
an $\C{O}$-lattice $\C{L}\subseteq\g$ of the form $\C{L}=\Lie\frak{G}$ for some split reductive group scheme
$\frak{G}$ over $\C{O}$ containing $\frak{S}$, whose generic fiber is $\G$.

Note that $\C{O}$-structures of $(\g,\S)$ exist. Moreover, any $\C{O}$-structure $\C{L}$ has a decomposition $\C{L}=\C{L}_0\oplus(\oplus_{\al\in\Phi(\G,\S)}\C{L}_{\al})$, where $\C{L}_{\al}\subseteq\u_{\al}$ is an $\C{O}$-lattice, and
$\C{L}_0=\Lie\frak{S}$.

(d) To every $\C{O}$-structure $\C{L}$ of $(\g,\S)$ and every affine root $\psi=\psi_{\al,k}\in\Psi(\G,\S)$, we associate an $\C{O}$-lattice
$\u_{\psi,\C{L}}:=\varpi^k\C{L}_{\al}\subseteq\u_{\al}$ and a subgroup $U_{\psi,\C{L}}=\iota_{\al}^{-1}(\u_{\psi,\C{L}})\subseteq U_{\al}$.

(e) Let $\C{L}$ and $\C{L}'$ be two $\C{O}$-structures of $(\g,\S)$. Since $\G$ is adjoint,  there exists an element $s\in \S(F)$ such that
$\Ad s(\C{L})=\C{L}'$. We denote by $\la=\la_{\C{L},\C{L}'}$ the image of $s$ in $X_*(\S)=\S(F)/\S(\C{O})$. Then $\la$ only depends on a pair $(\C{L},\C{L}')$ and
can be characterized as a unique element of $X_*(\S)$ such that $U_{\la^*(\psi),\C{L}}=U_{\psi,\C{L}'}$ for all $\psi\in\Psi(\G,\S)$.
\end{Emp}

\begin{Emp} \label{E:apart}
{\bf The apartment $\C{A}_{\S}$} (compare \cite[1.1]{Ti}).
(a) We denote by $\C{A}_{\S}$ the projective limit $\lim_{\C{L}}V_{\G,\S}$, where $\C{L}$ runs over the set of all $\C{O}$-structures of $(\g,\S)$,
and the transition maps are the isomorphisms $\la_{\C{L},\C{L}'}$ from \re{rig}(e). By construction, $\C{A}_{\S}$ is an affine space under $V_{\G,\S}$,
and we are given an affine isomorphism $\varphi_{\C{L}}:\C{A}_{\S}\isom V_{\G,\S}$
for every $\C{O}$-structure $\C{L}$. Moreover, $\C{A}_{\S}$ is equipped with a metric such that each $\varphi_{\C{L}}:\C{A}_{\S}\isom V_{\G,\S}$
is distance preserving.

(b) By construction,  $\C{A}_{\S}$ is  equipped with a set of affine roots
$\Psi(\C{A}_{\S})$ such that $\Psi(\C{A}_{\S})=\varphi_{\C{L}}^*(\Psi(\G,\S))$ for every $\C{O}$-structure $\C{L}$.
Moreover, for every affine root $\psi\in \Psi(\C{A}_{\S})$ with linear part $\al\in\Phi(\G,\S)$,
we are given a subgroup $U_{\psi}\subseteq U_{\al}$ such that $U_{\psi}=U_{(\varphi_{\C{L}})_*(\psi),\C{L}}\subseteq U_{\al}$ for every $\C{O}$-structure
$\C{L}$.

(c) For every $x\in \CA_{\S}$ we denote by $G_x\subseteq G$ the subgroup generated by $\S(\C{O})$ and the affine root subgroups $U_{\psi}$
taken over all $\psi$ such that $\psi(x)\geq 0$. It is called the {\em parahoric subgroup}, corresponding to $x$.
\end{Emp}

\begin{Emp} \label{E:simpl}
{\bf The simplicial decomposition of $\CA_{\S}$.}
(a) By the same formulas as in \re{poly}(b), the affine roots  $\Psi(\CA_\S)$ decompose $\CA_{\S}$ into simplices, thus giving to
$\CA_{\S}$ a structure of a simplicial complex.  (This is the only place, where  the assumption that $\G$ is simple is used).

(b) Let $x\in \CA_{\S}$, and let $\si$ be a unique simplex of $\CA_{\S}$ containing $x$.
Then $x$ can be written uniquely as a convex linear combination $x=\sum_{i}c_ix_i$, where $x_i$ runs through the set of all vertices of $\si$ (see \re{notcomb}(c)),
$0<c_i<1$ for all $i$ and $\sum_i c_i=1$.
\end{Emp}

 \begin{Emp}
{\bf The Bruhat--Tits building.}
(a) $\S,\S'\subseteq\G$ be maximal split tori, and let $x\in\CA_{\S}$ and $x'\in\CA_{\S'}$ be vertices. We say that $x\sim x'$, if the corresponding parahoric subgroups
(see \re{apart}(c)) are equal, that is, $G_x=G_{x'}$.

(b) Let $\S$ and $\S'$ be as in (a), and let $x\in\CA_{\S}$ and $x'\in\CA_{\S'}$ be arbitrary. We say that $x\sim x'$, if (after a permutation of vertices) the convex combinations $x=\sum_{i}c_ix_i$ and  $x'=\sum_{i}c'_ix'_i$ from \re{simpl}(b) satisfy $x_i\sim x'_i$ and $c'_i=c_i$ for all $i$. Clearly, $\sim$ is an equivalence relation.

(c) The (reduced) Bruhat--Tits building of $\G$ is the quotient of the disjoint union of apartments $\sqcup_{\S}\C{A}_{\S}$ by the equivalence relation $\sim$ defined in (b).
\end{Emp}

\section{Moy--Prasad filtrations} \label{S:mpf}

\noindent In this section we review the construction and basic properties of the Moy--Prasad filtrations (see \cite{MP1,MP2}) first in the split case, and then in general.

\begin{Emp} \label{E:split}
{\bf Filtration for split tori.} Let $\TT$ be a split torus, $\t:=\Lie\TT$, and $r\in\B{R}_{\geq 0}$.

(a) We denote by $T_r\subseteq T$ the subgroup of all $t\in T$ such that $\val_F(\la(t)-1)\geq r$ for every character $\la$ of $\TT$.
Similarly, we denote by $\t_r\subseteq\t$ the $\C{O}$-module consisting of all $a\in \t$ such that $\val_F(d\la(a))\geq r$ for every  $\la$.

(b) Note that $\TT$ has a natural structure over $\C{O}$, and $T_0=\TT(\C{O})\subseteq T$ is the maximal compact subgroup.
Moreover, let $n$ be the smallest integer such that $n\geq r$. Then $T_r$ is the kernel of the reduction map
$\TT(\C{O})\to \TT(\C{O}/(\varpi)^n)$. Similarly, $\t_0=\t(\C{O})$, and $\t_r$ is the kernel of the reduction map
$\t(\C{O})\to \t(\C{O}/(\varpi)^n)$.
\end{Emp}

\begin{Emp} \label{E:mpsplit}
{\bf Moy--Prasad filtrations for split groups.} Assume that $\G$ is split. Fix $x\in\C{X}$ and $r\geq 0$. Choose an apartment $\CA\subseteq\CX$
containing $x$, let $\S\subseteq\G$ be the corresponding maximal split torus, and set $\TT:=\Z_{\G}(\S)$ be the centralizer.

(a) Then $\TT=\S$, and the  Moy--Prasad subgroup $G_{x,r}\subseteq G$ is defined to be the subgroup, generated by $T_r$ and the affine root subgroups $U_{\psi}$,
where $\psi$ runs over all elements of $\Psi(\CA)$ such that $\psi(x)\geq r$.
Next, we denote by $\g_{x,r}\subseteq \g$ the $\C{O}$-submodule, spanned by $\t_r$ and $\u_{\psi}$
for all $\psi\in\Psi(\CA)$ with $\psi(x)\geq r$.

(b) Using property \re{buildings}(b) of the Bruhat--Tits buildings, one can show that both $G_{x,r}$ and $\g_{x,r}$ do not depend on a choice $\C{A}$.

(c) We set $G_{x,r^+}:=\cup_{s>r}G_{x,s}$, and $\g_{x,r^+}:=\cup_{s>r}\g_{x,s}$. Clearly, $G_{x,r^+}=G_{x,r'}$ and
$\g_{x,r^+}=\g_{x,r'}$ for some $r'>r$. We also denote by $\g^*_{x,-r}\subseteq\g^*$ the $\C{O}$-submodule, consisting of all
$b\in\g^*$ such that $\lan b,a\ran\in(\varpi)$ for every $a\in \g_{x,r^+}$.

(d) By definition, for every $x\in\CX$ the subgroup $G_{x,0}$ is the parahoric subgroup $G_x$ corresponding to $x$ (see \re{apart}(c)).
\end{Emp}




Next we define Moy--Prasad filtrations in general.

\begin{Emp} \label{E:filtor}
{\bf The Moy--Prasad filtration for tori.} Let $\TT$ be a torus over $F$.

(a) Let $\Gm_{\nr}:=\Gal(\ov{F}/F^{\nr})$. We let $w_{\TT}:\TT(F^{\nr})\to  X_*(\TT)_{\Gm_{\nr}}$ be the homomorphism, constructed by  Kottwitz (see \cite[Section 7]{Ko}), and set $T_0:=T\cap\Ker w_{\TT}$. Note that this definition coincides with that from \re{split} when $\TT$ is split.

(b) Let $F'/F$ be the splitting field of $\TT$, and let $e$ be the ramification degree of $F'/F$. We set $\TT':=\TT_{F'}$, $\t:=\Lie\TT$, and  $\t':=\Lie\TT'$.
Since $\TT'$ is split, the Moy--Prasad subgroups (resp. sublattices) of $T'$ (resp $\t'$) are defined (see \re{split}), and we set
$T_r:=T'_{re}\cap T_0$ and $\t_r:=\t'_{re}\cap \t$.
\end{Emp}

\begin{Emp}
{\bf Remark.}
Alternatively, $T_0\subseteq T$ can be defined as a group of $\C{O}$-points of the connected N\'eron model of $\TT$ (see \cite{HR}).
\end{Emp}

\begin{Emp} \label{E:mp}
{\bf Moy--Prasad filtrations in general.} Let $x$, $r$, $\CA$ and $\S$ be as in \re{mpsplit}.

(a) Assume that $\G$ is quasi-split. Then $\TT=\Z_{\G}(\S)$ is a maximal torus of $\G$, and we define
the subgroup $G_{x,r}\subseteq G$ and the $\C{O}$-submodule $\g_{x,r}\subseteq \g$ by the same formulas as in
the split case (see \re{mpsplit}) except that $T_r$ and $\t_r$ are defined in \re{filtor} instead of \re{split}, and
$U_{\psi}$ and $\u_{\psi}$ are defined in \re{root}(c) instead of \re{apart}(b).
As in the split case, both $G_{x,r}$ and $\g_{x,r}$ do not depend on a choice of $\C{A}$.

(b) For an arbitrary $\G$, let $F'/F$ be a finite unramified extension of minimal degree such that $\G':=\G_{F'}$ is quasi-split (see \rl{f'}). Then $G'_{x,r}$ and $\g'_{x,r}$ were defined in (a), and set $G_{x,r}:=G'_{x,r}\cap G$ and $\g_{x,r}:=\g'_{x,r}\cap \g$.

(c) We define $G_{x,r^+}$, $\g_{x,r^+}$ and  $\g^*_{x,-r}$ as in \re{mpsplit}(c). Then
$G_{x,r^+}\subseteq G_{x,r}$ is a normal subgroup.
\end{Emp}

\begin{Emp} \label{E:parah}
{\bf Subgroup $G^0\subseteq G$ and parahoric subgroups.} Let  $\G^{\sc}$ be the simply connected covering of the derived group of $\G$, and let $\iota:\G^{\sc}\to\G$ be the natural homomorphism.

(a) Assume that $\G$ is quasi-split, and let $\TT=\Z_{\G}(\S)$ be as in \re{mp}(a). Then $G^0:=T_0\cdot \iota(G^{\sc})\subseteq G$
is a normal subgroup of $G$, independent of $\S$.

(b) In general, we consider the unramified extension $F'/F$ as in \re{mp}(b). Then $(G')^0\subseteq G'$ is defined in (a), and
we set $G^0:=G\cap (G')^0$.

(c) Arguing as in \cite{HR} one can show that $G^0$ is equal to $G\cap\Ker w_\G\subset\G(F^{\nr})$, where $w_\G$ is the Kottwitz homomorphism (\cite[Section 7]{Ko}) for $\G$.

(d) By (c) and \cite{HR}, for each $x\in\C{X}$ the parahoric subgroup
$G_x:=G_{x,0}$ is equal to the stabilizer $\Stab_{G^0}(x)$ of $x$ in $G^0$.
\end{Emp}

\begin{Emp} \label{E:remmp}
{\bf Remarks.} (a) Let $F^{\flat}/F$  be a finite unramified extension, set $\Gm^{\flat}:=\Gal(F^{\flat}/F)$  and $\G^{\flat}:=\G_{F^{\flat}}$.
Then we have the equalities $G_{x,r}=(G^{\flat}_{x,r})^{\Gm^{\flat}}$ and $\g_{x,r}=(\g^{\flat}_{x,r})^{\Gm^{\flat}}$. Indeed, for $G_x$ the assertion follows from \re{parah}(d), while the remaining cases follow from \re{unr}(e) and \rl{iwahori}(b).

(b) Formally speaking, our definitions of $G_{x,r}$ and $g_{x,r}$ differ from the original definitions of Moy--Prasad. However, the two definitions are equivalent. Namely, the equivalence for $G_{x,r}$
can be shown by the same arguments as in (a), while the equivalence for $\g_{x,r}$ can be shown by the same argument as in \re{lattice}(a).

(c) It can be shown that every $\g_{x,r}\subseteq\g$ is a Lie subalgebra over $\C{O}$, but we are not going to use this fact.
\end{Emp}


The following property of Moy--Prasad filtrations was used in \re{tm} and \re{mplie}.

\begin{Lem} \label{L:mpfil}
For every $\si\in [\C{X}_m]$, $x,y\in\si$ and $r\in\frac{1}{m}\B{Z}$, we have
equalities $G_{x,r}= G_{y,r}$, $G_{x,r^+}= G_{y,r^+}$, $\g_{x,r}=\g_{y,r}$ and $\g_{x,r^+}=\g_{y,r^+}$.
\end{Lem}

\begin{proof}
Replacing $F$ by a finite unramified extension, if necessary, we may assume that $\G$ is quasi-split.
Choose an apartment $\C{A}\supseteq\si$. Then we have to show that for every  $\psi\in\Psi(\C{A})$, we have
$\psi(x)\geq r$ (resp. $\psi(x)>r$) if and only if $\psi(y)\geq r$ (resp. $\psi(y)>r$).
Since $\psi-r\in\Psi_m(\CA)$ (see \re{ref}(b)), this follows from the definition of the refined decomposition
(see \re{ref}(c)).
\end{proof}

\begin{Emp} \label{E:notmp}
{\bf Notation.}
Let $\S\subseteq\G$  be a maximal split torus. Then $\M:=\Z_\G(\S)$ is a minimal Levi subgroup of $\G$.
Thus $\M^{\ad}$ is anisotropic, hence the building $\CX(\M)$ is a single point $\{x_\M\}$. We set $\mm:=\Lie\M$ and define
$M_r:=M_{x_{\M},r}$ and $\mm_r:=\mm_{x_{\M},r}$.
\end{Emp}

The following basic property of Moy--Prasad filtrations follows from definitions when $\G$ is quasi-split, and it follows from Galois descent
(see \re{pfmp}) in general.

\begin{Prop} \label{P:mp}
Let $\S$ and $\M$ be as in \re{notmp}, set $\C{A}:=\C{A}_{\S}$, and choose $x\in \C{A}$ and  $r\in\B{R}_{\geq 0}$.

(a) The subgroup $G_{x,r}$ (resp. $G_{x,r^+}$) of $G$ is generated by  $M_{r}$ and the affine root subgroups $U_{\psi}$,
where $\psi$ runs over all elements of $\Psi(\CA)$ such that $\psi(x)\geq r$ (resp. $\psi(x)>r$).

(b) The $\C{O}$-module  $\g_{x,r}$ (resp. $\g_{x,r^+}$) of $\g$ is spanned by  $\mm_{r}$
and the $\C{O}$-submodules $\u_{\psi}$, where $\psi$ runs over all elements of $\Psi(\CA)$ such that $\psi(x)\geq r$ (resp. $\psi(x)>r$).
\end{Prop}

\begin{Lem} \label{L:convex}
For every $g\in G$ (resp. $a\in\g$, resp. $b\in \g^*$) and $r\in\B{R}_{\geq 0}$, the subset $\CX(g,r)$ (resp. $\CX(a,r)$, resp. $\CX(b,r)$) of  $\CX$, consisting of all $x\in\CX$ such that $g\in G_{x,r}$ (resp. $a\in \g_{x,r}$, resp. $b\in \g^*_{x,-r}$) is convex (see \re{notcomb}(b)).
\end{Lem}

\begin{Emp} \label{E:bpfconvex}
\begin{proof}[Proof of the Lemma] Here we only show the convexity of $\CX(g,0)$ and $\CX(b,r)$, while
the remaining assertions will be proven in \re{pfconvex}.

We have to show that for every $x,y\in\CX, z\in[x,y]$ and $r\in\B{R}_{\geq 0}$, we have inclusions
$G_{x}\cap G_{y}\subseteq G_{z}$ and $\g^*_{x,-r}\cap \g^*_{y,-r}\subseteq \g^*_{z,-r}$.  Choose an apartment $\CA$ in $\CX$ containing $x$ and $y$.

By \re{parah}(d), the inclusion $G_{x}\cap G_{y}\subseteq G_{z}$ can be rewritten as $\Stab_{G^0}(x)\cap\Stab_{G^0}(y)\subseteq\Stab_{G^0}(z)$.
Thus it suffices to show that for every $g\in G$,
the set of fixed points $\CX^g$ is convex. But this follows from the fact that the action of $G$ on $\CX$ is distance preserving
and that geodesics are unique.

Next, to show the inclusion $\g^*_{x,-r}\cap \g^*_{y,-r}\subseteq \g^*_{z,-r}$, it remains to show
the inclusion $\g_{z,r^+}\subseteq\g_{x,r^+}+\g_{y,r^+}$. By \rp{mp}(b), $\g_{z,r^+}$ is spanned by $\mm_{r^+}$ and $\u_{\psi}$, where $\psi$ runs over all elements of $\Psi(\CA)$ such that $\psi(z)>r$, and similarly for $x$ and $y$. Thus it suffices to show that for every $\psi\in\Psi(\CA)$ satisfying  $\psi(z)>r$, we have $\psi(x)> r$ or $\psi(y)> r$. But this follows from the assumption $z\in[x,y]$.
\end{proof}
\end{Emp}

The following result, whose proof will be given in \re{pfadler}, is a (slightly corrected) version of \cite[Prop 1.4.1]{Ad}. It implies that many questions about Moy--Prasad filtrations can be reduced to the split case. First we introduce a notation.

\begin{Emp}
{\bf ``Bad" groups.}
We say that $\G$ is ``bad", if $p=2$, and the group $\G^{\sc}_{F^{\nr}}$ has a factor $\R_{K/F^{\nr}}\SU{2n+1}$. Here $\G^{\sc}_{F^{\nr}}$ denotes the base change of $G^{\sc}$, $\R$ denotes the Weil restriction of scalars, and $\SU{2n+1}$ denotes the special unitary group.
\end{Emp}

\begin{Lem} \label{L:adler}
Assume that $\G$ is not ``bad". Let $F^{\flat}/F$ be a finite separable extension of ramification degree $e$.
Set $\G^{\flat}:=\G_{F^{\flat}}$, and $\g^{\flat}:=\Lie\G^{\flat}$.  Then for every $x\in\CX$ and $r\in\B{R}_{\geq 0}$ we have equalities $G_{x,r}=G^0\cap G^{\flat}_{x,re}$ and $\g_{x,r}=\g\cap \g^{\flat}_{x,re}$.
\end{Lem}

\section{Main technical result}

\begin{Emp} \label{E:notcomb}
{\bf Notation.}
(a) We define a partial order on $[\CX_m]$ by requiring that $\si'\preceq\si$, if $\si'$ is contained in the closure $\cl(\si)$ of $\si$.
In this case, we say that $\si'$ is a {\em face} of $\si$.

(b) We say that $\Si\subseteq[\CX_m]$ is a {\em subcomplex}, if the union $|\Si|:=\cup_{\si\in\Si}\si\subseteq\C{X}$ is closed.
Furthermore, we say that $\Si$ is {\em convex}, if $|\Si|$ is convex, that is, for every $x,y\in|\Si|$ the geodesic
$[x,y]$ in $\CX$ is also contained in $|\Si|$.

(c) By a {\em chamber} (resp.\,{\em vertex}) of $\CX_m$, we mean a polysimplex $\si\in[\CX_m]$ of maximal dimension (resp. dimension zero).
We denote the set of vertices of $\CX_m$ by  $V(\CX_m)$ and will not distinguish between a vertex $x\in V(\CX_m)$ and
the corresponding point of $\CX$. We say that $x\in V(\CX_m)$ is a vertex of $\si\in[\CX_m]$ if $x\preceq\si$.

(d) Let $\CA\subseteq\CX$ be  an apartment, $\psi\in\Psi_m(\CA)$ and $\si\in[\C{A}_m]$. We say that
$\psi(\si)>0$, if $\psi(y)>0$ for every $y\in\si$. Similarly, we define $\psi(\si)=0$, $\psi(\si)\geq 0$, etc.
\end{Emp}

\begin{Emp} \label{E:psi}
{\bf Notation.}
(a) Let $\CA\subseteq\CX$ be an apartment, and $\si\in [\CA_m]$ a chamber. Denote by $\Dt_{\CA}(\si)$ the set of
$\psi\in \Psi_m(\CA)$ such that $\psi(\si)>0$, and $\psi(\si')=0$ for some
 face $\si'\prec\si$ of codimension one. We call elements of $\Dt_{\CA}(\si)$ {\em simple affine roots, relative to $\si$}.

(b) For $x\in V(\CX_m)$ and $s\in\B{R}_{\geq 0}$, we denote by $\Upsilon_{x,s}$ the set of all chambers $\si\in [\CX_m]$ such that for
every apartment $\CA\subseteq\CX$ containing $\si$ and $x$  and every $\psi\in\Dt_{\CA}(\si)$ we have $\psi(x)\leq s$.
\end{Emp}

\begin{Emp} \label{E:rempsi}
{\bf Remark.} By \re{buildings}(b), a chamber $\si\in [\CX_m]$ belongs to $\Upsilon_{x,s}$ if $\psi(x)\leq s$
for some apartment $\CA\supseteq\si,x$ and every $\psi\in\Dt_{\CA}(\si)$. 
\end{Emp}

The following lemma will be proven in \re{pffin} below.

\begin{Lem} \label{L:fin}
For every $s\in \B{R}_{\geq 0}$ and $x\in V(\CX_m)$, the set $\Upsilon_{x,s}$ is finite, and $\Upsilon_{x,0}=\emptyset$.
\end{Lem}

\begin{Emp} \label{E:Ex1}
{\bf The $\SL{2}$-case.} Let $\G=\SL{2}$, and normalize the metric on $\CX$ (see \re{bt} and \re{rembt}(c)) such that every chamber $\si\in[\CX_m]$ has diameter one.

Then $\si\in \Upsilon_{x,s}$ if and only if $\si$ is contained in the ball $B(x,s)$ with center $x$ and radius $s$.
In particular, in this case remark \re{rempsi} and \rl{fin} are immediate.
\end{Emp}


\begin{Emp} \label{E:gm}
{\bf The basic subcomplex.} Fix $\si'\in [\CX_m]$, $x\in V(\CX_m)$ and $s\in\B{R}_{\geq 0}$,
and choose an apartment $\CA\subseteq\CX$ containing
$\si',x$ (see \re{poly}(d)). We denote by $\Gm_s(\si',x)\subseteq[\CX_m]$ the subcomplex consisting of all $\si\in [\CA_m]$
such that for every $\psi\in\Psi_m(\CA)$ satisfying $\psi(\si')\leq 0$ and $\psi(x)\leq s$, we have $\psi(\si)\leq 0$.
\end{Emp}

\begin{Emp} \label{E:remgm}
{\bf Remarks.}
(a) By \re{buildings}(b), the subcomplex  $\Gm_s(\si',x)$ does not depend on the choice of $\CA$. Namely, this
follows from the fact that an isomorphism $\CA\isom\CA'$ from \re{buildings}(b) induces a bijection
$\Psi_m(\CA')\isom\Psi_m(\CA)$ on refined affine roots.

(b) Note that $\Gm_0(\si',x)$ is the smallest convex subcomplex of $[\CX_m]$ containing $\si'$ and $x$.
This subcomplex was studied in \cite{MS}.

(c) By definition, the complex $\Gm_s(\si',x)$ is convex, and $\Gm_s(\si',x)\subseteq\Gm_0(\si',x)$.

(d) For every $\si\in \Gm_s(\si',x)$ and $\si''\in \Gm_s(\si,x)$, we have $\si''\in \Gm_s(\si',x)$.
\end{Emp}

\begin{Emp} \label{E:Ex2}
{\bf The $\SL{2}$-case.} In the situation of \re{Ex1}, let $\si'=y$ be a vertex, and $\si\in[\CX_m]$ a chamber. Then $\si\in \Gm_0(\si',x)$ if and only if
$\si\subseteq[y,x]$. More generally, $\si\in \Gm_s(\si',x)$ if and only if $\si\subseteq[y,x]$ and $\si\nsubseteq B(x,s)$.
\end{Emp}

The complex $\Gm_s(\si',x)$ is important to us because of the following fact.

\begin{Lem} \label{L:equal}
Let $\si,\si'\in [\CX_m]$, $x\in V(\CX_m)$ and $r,s\in\frac{1}{m}\B{Z}_{\geq 0}$ such that $\si'\preceq\si$ and
$\si\in\Gm_s(\si',x)$. Then we have the equality
$\dt_{G_{\si,r^+}}\ast\dt_{G_{x,(r+s)^+}}=\dt_{G_{\si',r^+}}\ast\dt_{G_{x,(r+s)^+}}$.
\end{Lem}

\begin{proof}
By definition, $\dt_{G_{\si,r^+}}\ast\dt_{G_{x,(r+s)^+}}$ is the pushforward of $\dt_{(G_{\si,r^+}\times G_{x,(r+s)^+})}$ under the multiplication map
$G\times G\to G$. Therefore $\dt_{G_{\si,r^+}}\ast\dt_{G_{x,(r+s)^+}}$ can be characterized as a unique $G_{\si,r^+}\times G_{x,(r+s)^+}$-invariant measure on $G$, supported on $G_{\si,r^+}\cdot G_{x,(r+s)^+}$ with total measure one. This also holds with $\si$ replaced by $\si'$.

Since $\si'\preceq\si$, we have $G_{\si',r^+}\subseteq G_{\si,r^+}$. Therefore it suffices to check the equality of sets $G_{\si,r^+}\cdot G_{x,(r+s)^+}=G_{\si',r^+}\cdot G_{x,(s+r)^+}$, or, equivalently, the inclusion
\begin{equation} \label{Eq:incl}
G_{\si,r^+}\subseteq G_{\si',r^+}\cdot (G_{\si,r^+}\cap G_{x,(r+s)^+}).
\end{equation}

Choose an apartment $\CA=\CA_\S\subseteq \CX$ containing $\si$ and $x$. It follows from the definition of Moy--Prasad subgroups in the split case and from \rp{mp}(a) in  general that the subgroup $G_{\si,r^+}$ is generated by $M_{r^+}$ and the affine root subgroups $U_{\psi}$, where $\psi$ runs over elements of $\Psi(\C{A})$ such that $\psi(\si)>r$. The same also holds for $G_{\si',r^+}$ and $G_{x,(r+s)^+}$.

Since $G_{\si',r^+}\subseteq G_{\si,r^+}\subseteq G_{\si,r} \subseteq G_{\si',r}$, and $G_{\si',r^+}\subseteq G_{\si',r}$ is a normal subgroup,
the right-hand side of \form{incl} is a group. Thus it suffices to show that for every $\psi\in\Psi(\CA)$ satisfying
$\psi(\si)>r$, we have $\psi(\si')>r$ or $\psi(x)>r+s$.

For every $\psi\in\Psi(\CA)$ we have $\psi-r\in \Psi_m(\CA)$ (see \re{ref}(b)).  Replacing $\psi$ by $\psi-r$, it suffices to show that for every $\psi\in\Psi_m(\CA)$ satisfying $\psi(\si)>0$, we have $\psi(x)>s$ or $\psi(\si')>0$, which is equivalent to the assumption $\si\in\Gm_s(\si',x)$.
\end{proof}

The following result (and its proof) is a generalization of \cite[Lem. 2.8, 2.9]{MS}, where the case $s=0$ is studied. It will be proved in Section \re{pfmin} below.

\begin{Lem} \label{L:min}
Let $x\in V(\CX_m)$, $s\in\B{R}_{\geq 0}$ and $\si\in [\CX_m]$.

(a) There exists a unique minimal face $\si'=m_{x,s}(\si)$ of $\si$ such that $\si\in \Gm_s(\si',x)$.

(b) There exists a unique maximal polysimplex $\si''\in\Gm_s(\si',x)$ such that $\si'\preceq\si''$.
\end{Lem}

\begin{Emp} \label{E:not}
{\bf Notation.}
For $x\in V(\CX_m)$ and $s\in\B{R}_{\geq 0}$, we denote by $m_{x,s}:[\CX_m]\to [\CX_m]$ the map defined in \rl{min}(a). It is idempotent by \re{remgm}(d).
\end{Emp}

\begin{Emp} \label{E:Ex3}
{\bf The $\SL{2}$-case.} Assume that we are in the situation of \re{Ex1}.

(a) Let $\si\in[\CA_m]$ be a chamber. Using the description of \re{Ex2}, one sees that $y:=m_{x,0}(\si)$ is the unique vertex of $\si$ such that $\si\subseteq[x,y]$.
Moreover, we have $m_{x,s}(\si)=y$ if $d(x,y)>s$, and $m_{x,s}(\si)=\si$ otherwise.

(b) Let $\si'=y$ be a vertex. Then $\si''$ is the unique chamber $\si\subseteq[x,y]$ such that $y\preceq\si$ if $d(x,y)>s$, and $\si''=\si'$ otherwise.
\end{Emp}

The following lemma will be proved in Section \re{pfequiv} below.

\begin{Lem} \label{L:equiv}
Let $x,s,\si,\si'$ and $\si''$ be as in \rl{min}, and let $\tau\in [\CX_m]$.

(a) We have $\si'\preceq\tau\preceq\si''$ if and only if $m_{x,s}(\tau)=\si'$.

(b) Let $\Si,\Si'\in\T_m$ (see \re{tm}(d)) be such that $x\in \Si'\subseteq\Si$ and $\si\in\Si\sm\Si'$. Then
for every $\tau$ satisfying $\si'\preceq\tau\preceq\si''$ we have $\tau\in\Si\sm\Si'$.

(c) In the situation of (b) assume that  $\Si'\supseteq\Upsilon_{x,s}$. Then $\si''\neq \si'$.
\end{Lem}

Now we are ready to prove our main technical result.

\begin{Prop} \label{P:stab}
(a) Let $x\in V(\CX_m)$, $r,s\in\frac{1}{m}\B{Z}_{\geq 0}$ and let $\Si,\Si'\in\T_m$ be such that
$x\in \Si'\subseteq \Si$ and $\Upsilon_{x,s}\subseteq\Si'$, and let $E^{\Si}_r$ be as in \re{tm}(c). Then we have the equality
\[E^{\Si}_r\ast\dt_{G_{x,(r+s)^+}}= E^{\Si'}_r\ast\dt_{G_{x,(r+s)^+}}.\]

(b)  For every $r\in\frac{1}{m}\B{Z}_{\geq 0}$, $\Si\in\T_m$ and $\si\in\Si$, we have  $E^{\Si}_r\ast\dt_{G_{\si,r^+}}= \dt_{G_{\si,r^+}}$.
\end{Prop}

\begin{proof}
(a) Setting $\Si'':=\Si\sm\Si'$, our assertion can be rewritten as $E^{\Si''}_r\ast\dt_{G_{x,(r+s)^+}}=0$.
Let $m_{x,s}$ be as in \re{not}, and define an equivalence relation on $\Si''$ by requiring that
$\si_1\sim\si_2$ if and only if $m_{x,s}(\si_1)=m_{x,s}(\si_2)$. For every
$\si\in\Si''$, we denote by $\Si''_{\si}\subseteq\Si''$ the equivalence class of $\si$.
Then $\Si''$ decomposes
as a disjoint union of the $\Si''_{\si}$'s. Thus it suffices to show that
$E^{\Si''_{\si}}_r\ast\dt_{G_{x,(r+s)^+}}=0$ for every $\si\in\Si''$.

By \rl{equal}, for every $\tau\in [\CX_m]$ we have
\[
\dt_{G_{\tau,r^+}}\ast\dt_{G_{x,(r+s)^+}}=\dt_{G_{m_{x,s}(\tau),r^+}}\ast\dt_{G_{x,(r+s)^+}}.\] Since every
$\tau\in\Si''_{\si}$ satisfies $m_{x,s}(\tau)=m_{x,s}(\si)$, we have
\[
E^{\Si''_{\si}}_r\ast\dt_{G_{x,(r+s)^+}}=\left(\sum_{\tau\in \Si''_{\si}}(-1)^{\dim \tau}\right) (\dt_{G_{m_{x,s}(\si),r^+}}\ast\dt_{G_{x,(r+s)^+}}).
\]
Thus it remains to show that $\sum_{\tau\in \Si''_{\si}}(-1)^{\dim \tau}=0$.

Let $\si',\si''\in [\CX_m]$ be as in \rl{min}. By \rl{equiv}(a),(b), the equivalence class $\Si''_{\si}\subseteq\Si''$ consists of all $\tau$ such that
$\si'\preceq\tau\preceq\si''$. Thus the sum $\sum_{\tau\in \Si''_{\si}}(-1)^{\dim \tau}$ equals  $\sum_{\tau,\si'\preceq\tau\preceq\si''}(-1)^{\dim \tau}$,
and the latter expression vanishes, because $\si''\neq\si'$ (by \rl{equiv}(c)).

(b) Choose $x\in V(\CX_m)$ such that $x\preceq\si$. Then $G_{x,r^+}\subseteq G_{\si,r^+}$, hence we have
$\dt_{G_{x,r^+}}\ast \dt_{G_{\si,r^+}}=\dt_{G_{\si,r^+}}$. Thus it suffices to show that $E_r^{\Si}\ast \dt_{G_{x,r^+}}=\dt_{G_{x,r^+}}$.

Since $\Upsilon_{x,0}$ is empty (by \rl{fin}), the subcomplex $\Si':=\{x\}$ satisfies the assumptions of (a) with $s=0$. Thus, by (a), we have
\[
E_r^{\Si}\ast \dt_{G_{x,r^+}}=E_r^{\{x\}}\ast \dt_{G_{x,r^+}}=\dt_{G_{x,r^+}}\ast \dt_{G_{x,r^+}}=\dt_{G_{x,r^+}},
\]
and the proof is complete.
\end{proof}

\section{Combinatorics of the building}

\noindent In this section we prove Lemmas \ref{L:fin}, \ref{L:min} and \ref{L:equiv}. Replacing $\G$ by $\G^{\ad}$, we can assume that $\G$ is adjoint.

\begin{Emp} \label{E:pffin}
\begin{proof}[Proof of \rl{fin}]
Fix a chamber $\si\in[\C{X}_m]$ and an apartment $\C{A}$ containing $\si$ and $x$. Decomposing $\G$ and $\CX(\G)$ into a product, if necessary, we may assume that
$\G$ is simple. Then there exist positive numbers $\{n_{\psi}\}_{\psi\in\Dt_{\C{A}}(\si)}$ such that the affine function $\sum_{\psi}n_{\psi}\psi$ is $1$. Indeed,
this is standard for $m=1$, and the general case follows from it.
Since in the linear combination  $\sum_{\psi\in\Dt_{\CA}(\si)}n_{\psi}\psi(x)=1>0$ we have $n_{\psi}>0$ for all $\psi$, there exists $\psi\in\Dt_{\CA}(\si)$ such that $\psi(x)>0$. Hence $\si\notin\Upsilon_{x,0}$. Since $\si$ was arbitrary, we conclude that $\Upsilon_{x,0}=\emptyset$.

Note that the parahoric subgroup $G_x$ acts transitively on the set of apartments containing $x$, the set $\Upsilon_{x,s}$ is $G_x$-invariant, and the polysimplicial complex $[\CX]$ is locally finite. Therefore we can fix an apartment $\CA=\CA_{\S}\ni x$, and it suffices to show that the intersection $[\CA_m]\cap \Upsilon_{x,s}$ is finite.

For every chamber $\si\in[\CA_m]\cap \Upsilon_{x,s}$, point $y\in\si$ and affine root $\psi\in \Dt_{\CA}(\si)$, we have
$\psi(x)\leq s$ and $\psi(y)>0$.  Hence the difference $x-y\in V_{\G,\S}$ satisfies $\al_{\psi}(x-y)< s$ (see \re{root}(b)) for all $\psi\in \Dt_{\C{A}}(\si)$.
From this we conclude that $x-y$ lies in a bounded set; thus the intersection $[\CA_m]\cap \Upsilon_{x,s}$ is finite.
\end{proof}
\end{Emp}

\rl{min}(a) will be deduced from the following more precise result.

\begin{Lem} \label{L:min'}
Fix $x\in V(\CX_m)$, $s\in\B{R}_{\geq 0}$, $\si\in [\CX_m]$,
and let $\CA\subseteq\CX$ be an apartment, containing $x$ and $\si$.

(a) Then there exists a chamber $\wt{\si}\in [\CA_m]$
such that $\si\preceq\wt{\si}$ and for every  $\psi\in\Dt_{\CA}(\wt{\si})$ with $\psi(\si)=0$,
we have $\psi(x)\geq 0$.

(b) Assume that $\si\neq x$, and  $\G$ is simple. Then there exists a unique minimal face $\si'$ of $\si$ such that
$\si\in \Gm_s(\si',x)$. Moreover,  $\si'$ is characterized by the condition
that for every $\psi\in\Dt_{\CA}(\wt{\si})$ we have $\psi(\si')=0$ if and only if $\psi(\si)=0$ or $\psi(x)>s$.
\end{Lem}
\begin{proof}
(a)  Choose a point $y\in\si$, and a chamber $\wt{\si}\in [\CA_m]$ such that $\si\preceq\wt{\si}$ and
$\cl(\wt{\si})\cap (y,x]\neq\emptyset$. We claim that this chamber satisfies the required property.
Indeed, let $\psi\in\Psi_m(\CA)$ be such that $\psi(\si)=0$ and $\psi(x)< 0$. Then $\psi(y)=0$, thus
$\psi|_{(y,x]}<0$. Since $\cl(\wt{\si})\cap (y,x]\neq\emptyset$, we conclude that $\psi(\wt{\si})<0$,
hence $\psi\notin\Dt_{\CA}(\wt{\si})$.

(b) Assume that for every $\psi\in \Dt_{\CA}(\wt{\si})$ we have
$\psi(\si)=0$ or $\psi(x)>s$. Then, by our choice of $\wt{\si}$, for  every $\psi\in \Dt_{\CA}(\wt{\si})$ we have
$\psi(x)\geq 0$; thus  $x\preceq\wt{\si}$. Since $\si\neq x$, there exists $\psi_0\in \Dt_{\CA}(\wt{\si})$ such that
$\psi_0(x)=0$ and $\psi_0(\si)> 0$. Then $\psi_0(x)\leq s$, contradicting our assumption.

By the previous paragraph, there exists a unique face $\si'\preceq\si$ such that for every $\psi\in \Dt_{\CA}(\wt{\si})$ we have $\psi(\si')=0$ if and only if $\psi(\si)=0$ or $\psi(x)>s$. We claim that $\si\in \Gm_s(\si',x)$, that is, for every $\xi\in\Psi_m(\CA)$ satisfying $\xi(\si')\leq 0$ and $\xi(\si)>0$, we have  $\xi(x)>s$.

Since $\xi(\si)>0$, we have $\xi(\wt{\si})>0$. Thus the affine root $\xi$ is of the form $\sum_{\psi\in \Dt_{\CA}(\wt{\si})}n_{\psi}\psi$, where $n_{\psi}\in\B{Z}_{\geq 0}$ for all $\psi$. Since $\xi(\si')\leq 0$, we get $n_{\psi}=0$ when $\psi(\si')>0$. Thus every $\psi\in \Dt_{\CA}(\wt{\si})$ with $n_{\psi}>0$ satisfies $\psi(\si')=0$, that is, $\psi(\si)=0$ or $\psi(x)>s$. By our choice of $\wt{\si}$ (see (a)), in both cases, we have $\psi(x)\geq 0$.

Since $\xi(\si)>0$, there exists therefore $\psi_0\in \Dt_{\CA}(\wt{\si})$ with $\psi_0(x)>s$ and $n_{\psi_0}>0$. Hence
$\xi(x)\geq n_{\psi_0}\psi_0(x)\geq \psi_0(x)>s$, as claimed.

It remains to show that for every $\si''\preceq\si$ such that $\si\in \Gm_s(\si'',x)$ we have $\si'\preceq\si''$.
Choose $\psi\in \Dt_{\CA}(\wt{\si})$ such that $\psi(\si'')=0$. We want to show that $\psi(\si')=0$, that is,
$\psi(\si)=0$ or $\psi(x)>s$.

Equivalently, assuming that $\psi(x)\leq s$, we want to conclude that $\psi(\si)=0$, that is,
$\psi(\si)\leq 0$ and $\psi(\si)\geq 0$. Since $\si\in \Gm_s(\si'',x)$ and $\psi(\si'')\leq 0$, we have
$\psi(\si)\leq 0$. On the other hand, since $\psi\in \Dt_{\CA}(\wt{\si})$ and $\si\preceq\wt{\si}$,
we have $\psi(\si)\geq 0$.
\end{proof}

\begin{Emp} \label{E:pfmin}
\begin{proof}[Proof of \rl{min}]

(a) If $\si=x$, then $\si':=x$ satisfies the property, so we can assume that $\si\neq x$. Decomposing $\G$ as a product $\prod_i\G_i$, we get a decomposition of   $[\CX_m(\G)]$ as the product $\prod_i [\CX_m(\G_i)]$. Then $\si$ and $x$ decompose as  products $\si=\prod\si_i$ and $x=\prod x_i$.
Moreover, every face $\si'\preceq\si$ decomposes as $\si'=\prod_i\si'_i$, and we have
$\si\in \Gm_s(\si',x)$ if and only if $\si_i\in \Gm_s(\si'_i,x_i)$ for all $i$. Thus we can assume that
$\G$ is simple, in which case the assertion follows from \rl{min'}(b).

(b) Consider two maximal polysimplices $\si''_1,\si''_2\in \Gm_s(\si',x)\subseteq[\C{A}_m]$ such that $\si'\preceq\si''_1,\si''_2$.
First we claim that $\si''_1$ and $\si''_2$ are faces of the same chamber. For this we have to show that there is no
$\psi\in\Psi_m(\CA)$ such that $\psi(\si''_1)>0$ and  $\psi(\si''_2)< 0$.

Indeed, assume that there exists $\psi\in\Psi_m(\CA)$ such that $\psi(\si''_1)>0$ and  $\psi(\si''_2)< 0$.
Since $\si'\preceq\si''_2$ and $\si''_1\in \Gm_s(\si',x)$, this implies that $\psi(\si')\leq 0$, thus $\psi(x)>s\geq 0$.
Similarly, repeating the above argument interchanging $\si''_1$ with $\si''_2$ and $\psi$ with $-\psi$, we conclude that
$\psi(x)< 0$, a contradiction.

Since $\si''_1$ and $\si''_2$ are faces of the same chamber, they generate a polysimplex $\si''_3$ such that $\si''_1,\si''_2\preceq\si''_3$.
Moreover, since $\Gm_s(\si',x)$ is convex, we conclude that $\si''_3\in\Gm_s(\si',x)$. Since $\si''_1$ and $\si''_2$ are assumed to be
maximal, we thus conclude that $\si''_1=\si''_3=\si''_2$.
\end{proof}
\end{Emp}

\begin{Emp} \label{E:pfequiv}
\begin{proof}[Proof of \rl{equiv}]
(a)  Assume that $m_{x,s}(\tau)=\si'$. Then $\si'\preceq\tau$ and $\tau\in \Gm_s(\si',x)$.
Hence $\tau\preceq\si''$ by the definition of $\si''$ (see \rl{min}(b)).

Conversely, assume that $\si'\preceq\tau\preceq\si''$, and we want to show that $\tau':=m_{x,s}(\tau)$ equals $\si'$. Since $\si''\in \Gm_s(\si',x)$ and $\tau\preceq\si''$, we conclude that $\tau\in \Gm_s(\si',x)$, thus $\tau'\preceq\si'$. On the other hand, since $\tau\in \Gm_s(\tau',x)$ and $\si'\preceq\tau$, we have
$\si'\in \Gm_s(\tau',x)$. Since $\tau'\preceq\si'$, we conclude that $m_{x,s}(\si')\preceq\tau'\preceq\si'$. Finally, since $m_{x,s}(\tau)=\si'$, we deduce that
$m_{x,s}(\si')=\si'$, thus $\tau'=\si'$.

(b) Assume that $\si'\preceq\tau\preceq\si''$, and we want to show that $\tau\in \Si$ and $\tau\notin \Si'$. Since $\Si'$ and $\Si$ are subcomplexes, it suffices to show that $\si'\notin\Si'$ and
$\si''\in\Si$.

Assume that $\si'\in\Si'$. Since $x\in\Si'$ and $\Si'$ is convex, we conclude that $\Gm_0(\si',x)\subseteq\Si'$ (see \re{remgm}(b)). Thus $m_{x,s}^{-1}(\si')\subseteq \Gm_s(\si',x)\subseteq \Gm_0(\si',x)$ (see \re{remgm}(c)) is contained in $\Si'$. But this contradicts the assumptions
$\si\in m_{x,s}^{-1}(\si')$ and $\si\notin\Si'$.

Next,  since $\si\in \Si$, $\si'\preceq\si$  and $\Si$ is a subcomplex, we conclude that $\si'\in \Si$. Thus, arguing as in the previous paragraph we conclude that  $\Gm_s(\si',x)\subseteq \Si$, thus $\si''\in\Si$.

(c) We have to show that there exists $\tau\neq \si'$ such that $\si'=m_{x,s}(\tau)$. Decomposing $\G$ into a product, if necessary, we may assume that
$\G$ is simple. Since $x\in\Si'$ and $\si'\notin \Si'$ (use (b)), we conclude that $\si'\neq x$.

Let $\CA\subseteq\CX$ be an apartment containing $\si'$ and $x$, and let $\wt{\si}\in[\CA_m]$ be a chamber such that $\si'\preceq\wt{\si}$ and
$\phi(x)\geq 0$ for every $\psi\in\Dt_{\CA}(\wt{\si})$ such that  $\psi(\si')=0$ (see \rl{min'}(a)).
Since $\si'\preceq\wt{\si}$ and $\si'\notin\Si'$, we conclude that $\wt{\si}\notin\Si'$.
Using the assumption $\Upsilon_{x,s}\subseteq\Si'$, we conclude that $\wt{\si}\notin  \Upsilon_{x,s}$. Thus, by Remark \re{rempsi}, there exists
$\psi_0\in\Dt_{\CA}(\wt{\si})$ such that $\psi_0(x)>s$.

By \rl{min'}(b), we conclude that $\psi_0(\si')=0$. Hence there exists a unique $\tau\preceq\wt{\si}$ such that $\si'$ is a face of $\tau$ of codimension one, and $\psi_0(\tau)>0$. By construction, for every $\psi\in\Dt_{\CA}(\wt{\si})$ with $\psi(\si')=0$ we have either $\psi(\tau)=0$ or $\psi=\psi_0$.
Since $\psi_0(x)>s$, the desired equality $m_{x,s}(\tau)=\si'$ follows the characterization of $m_{x,s}(\tau)$, given in \rl{min'}(b).
\end{proof}
\end{Emp}

\section{Formula for the projector and applications}
\noindent In this section we prove \rt{explicit}, \rp{lie}, and Corollaries \ref{C:explicit}, \ref{C:formchar}, \ref{C:Av} and \ref{C:lie}.

\begin{Emp} \label{E:pfexplicit}
\begin{proof}[Proof of \rt{explicit}] We divide the proof into six steps.

{\bf Step 1.}  For every $h\in \C{H}(G)$,  the inductive system  $\{E^{\Si}_r\ast h\}_{\Si\in\T_m}$ stabilizes.

\begin{proof}
Fix  $x\in V(\CX_m)$ and $n\in\B{N}$ such that $\dt_{G_{x,n^+}}\ast h=h$. It suffices to show that the inductive system  $\{E^{\Si}_r\ast\dt_{G_{x,n^+}}\}_{\Si\in\T_m}$ stabilizes, so the assertion follows from \rp{stab}(a).
\end{proof}

{\bf Step 2.} There exists a unique element $z\in Z_G$ such that $z(h)=E^{\Si}_r\ast h$ for every $h\in\C{H}(G)$ and every sufficiently large $\Si\in\T_m$, that is, $z(h)=\lim_{\Si\in \T_m}(E^{\Si}_r\ast h)$.

\begin{proof}
By Step 1,  there exists a unique endomorphism $z\in\End_{\B{C}}\C{H}(G)$ such that
$z(h)=\lim_{\Si\in\T_m} E^{\Si}_r\ast h$ for every $h\in\C{H}(G)$. We claim that $z\in Z_G$.

Since $z$ commutes with the right convolutions, it suffices to show that $z$ is $\Ad G$-invariant (use \re{cent}(b)).
First we claim that $z$ is $\Ad K$-invariant for every compact subgroup $K\subseteq G^{\ad}$. Indeed, the
$\Si\in\T_m$'s in the equality $z(h)=\lim_{\Si\in\T_m}(E^{\Si}_r\ast h)$ can be chosen to be $\Ad K$-invariant, thus $z$ is $\Ad K$-invariant.

It remains to show that the group $G^{\ad}$ is generated by compact subgroups. Since the corresponding simply connected group $G^{\sc}$ is known to be generated by compact subgroups, and $G^{\sc}$ acts transitively on the set of chambers in $[\CX(G)]$, the  assertion follows from the fact that the stabilizer $\Stab_{G^{\ad}}(\si)$ of every chamber is compact.
\end{proof}

{\bf Step 3.}  For every $V\in R(G)$ and $v\in V$, the inductive system  $\{E^{\Si}_r(v)\}_{\Si\in\T_m}$ stabilizes, and
 $z(v)=\lim_{\Si\in \T_m}E^{\Si}_r(v)$.

 \begin{proof}
 Choose $h\in \C{H}(G)$ such that $h(v)=v$. Then $E^{\Si}_r(v)=E^{\Si}_r(h(v))=(E^{\Si}_r\ast h)(v)$ stabilizes (by Step 1), and
 the limit value equals $z(h)(v)=z(h(v))=z(v)$ (see \re{cent}(c)).
\end{proof}

{\bf Step 4.} For every $V\in \Irr(G)_{\leq r}$, we have  $z|_{V}=\Id_V$.

\begin{proof}
By definition, there exists $x\in\CX$ such
that $V^{G_{x,r^+}}\neq 0$. Thus, by Schur's lemma, it remains to show that
$z(v)=v$ for all $v\in V^{G_{x,r^+}}$. Using \rp{stab}(b), we conclude that $z(\dt_{G_{x,r^+}})=\dt_{G_{x,r^+}}$. Note that for each  $v\in V^{G_{x,r^+}}$ we have $\dt_{G_{x,r^+}}(v)=v$. Therefore, by \re{cent}(c), we conclude that
$z(v)=z(\dt_{G_{x,r^+}}(v))=(z(\dt_{G_{x,r^+}}))(v)=\dt_{G_{x,r^+}}(v)=v.$
\end{proof}

{\bf Step 5.} For every $V\in R(G)_{>r}$, we have $z|_{V}=0$.

\begin{proof}
For every $V\in R(G)_{>r}$ and $x\in\CX$, we have $V^{G_{x,r^+}}=0$. Thus $\dt_{G_{x,r^+}}(v)=0$ for all $v\in V$.
Therefore we have $E^{\Si}_r(v)=0$ for all $\Si\in\T_m$ and $v\in V$, hence $z(v)=0$ by Step 3.
\end{proof}

{\bf Step 6.} Since an element of $Z_G$ is determined by its action on irreducible representations, it follows from
 Steps 4 and 5 that $z=\Pi_r$ (see \re{idemp} for a more direct argument).
\end{proof}
\end{Emp}

\begin{Emp} \label{E:idemp}
{\bf An alternative proof.} Using the arguments, described above, we can give both an alternative  proof of
the decomposition $R(G)=R(G)_{\leq r}\oplus R(G)_{>r}$ and a more direct proof of the equality $z=\Pi_r$. We do it in two steps.

(I) The element $z\in Z_G$, constructed in Step 2 of \re{pfexplicit}, is idempotent.

\begin{proof}
We have to show that $z\circ z=z$. By \re{cent}(b), it suffices to show that for every $h\in\C{H}(G)$ we have $z(z(h))=z(h)$. By
the definition of $z$, we have to show that $z(E_r^{\Si}\ast h)=E_r^{\Si}\ast h$ for all sufficiently large $\Si\in\T_m$.
By construction, we have $z(E_r^{\Si}\ast h)=z(E_r^{\Si})\ast h$. So it suffices to show that
 $z(E_r^{\Si})=E_r^{\Si}$ for every $\Si\in\T_m$, or equivalently that
 $z(\dt_{\si,r^+})=\dt_{\si,r^+}$ for every $\si\in [\CX_m]$. But this follows from the definition of $z$ and \rp{stab}(b).
\end{proof}

(II) For every  $V\in R(G)$, set $V_{\leq r}:=\im (z|_V)\subseteq V$ and  $V_{>r}:=\Ker (z|_V)\subseteq V$. Since $z\in Z_G$ is an idempotent, we have a direct sum decomposition $V=V_{\leq r}\oplus V_{>r}$, and we also have $z|_{W}=\Id_W$ (resp. $z|_{W}=0$) for every irreducible subquotient $W$ of $V^{\leq r}$ (resp. $V^{>r}$). Then the result of Step 5 (resp. Step 4) of \re{pfexplicit} implies that $V_{\leq r}\in R(G)_{\leq r}$ (resp. $V_{>r}\in R(G)_{>r}$).  This implies both the desired decomposition $R(G)=R(G)_{\leq r}\oplus R(G)_{>r}$ and the desired equality $z=\Pi_r$.
\end{Emp}

\begin{Emp} \label{E:pfexplicit}
\begin{proof}[Proof of \rco{explicit}]
(a) For  $f\in C_c^{\infty}(G)$ and $E\in D(G)$, we define the convolution $E\ast f\in C^{\infty}(G)$ by the rule $(E\ast f) dg:= E\ast (f dg)$ for a Haar measure $dg$ on $G$. Then $E(f)=(E\ast\iota^*(f))(1)$, where $\iota:G\to G$ is the map $g\mapsto g^{-1}$.

By \rt{explicit}, for every $h\in\C{H}(G)$ we have  $E_r\ast h=\lim_{\Si\in\T_m}(E^{\Si}_r\ast h)$. Therefore for every $f\in C_c^{\infty}(G)$ we have  $E_r\ast f=\lim_{\Si\in\T_m}(E^{\Si}_r\ast f)$, hence $E_r(f)=\lim_{\Si}E_r^{\Si}(f)$.

(b) Since each $E_r^{\Si}$ is supported on $G_{r^+}$, we conclude by (a).
\end{proof}
\end{Emp}

\begin{Emp}
{\bf Generalized functions of depth $\leq r$.}
%
(a) Since the space of generalized functions $\wh{C}(G)$ is the linear dual  of $\C{H}(G)$, the Bernstein center $Z_G$ acts on  $\wh{C}(G)$
by the formula $z(\chi)(h):=\chi(z(h))$ for every $z\in Z_G$, $\chi\in\wh{C}(G)$ and $h\in\C{H}(G)$.
We say that $\chi\in \wh{C}(G)$ is of {\em depth $\leq r$}, if $\Pi_r(\chi)=\chi$.

(b) Note that for every admissible representation $V\in R(G)_{\leq r}$, its character $\chi_{V}$ is of depth $\leq r$. Indeed, for every $h\in\C{H}(G)$ we have $\chi_{V}(h)=\Tr(h|_V)$ and
$\Pi_r(\chi_{V})(h)=\chi_{V}(\Pi_r(h))=\Tr(\Pi_r(h)|_V)$. Since $\Pi_r(h)|_V=\Pi_r|_V\circ h|_V$ (by \re{cent}(c)) and $\Pi_r|_V=\Id_V$ (because $V\in R(G)_{\leq r}$), the equality $\Pi_r(\chi_{V})=\chi_V$ follows.

Thus the following result is a generalization of \rco{formchar}.
\end{Emp}

\begin{Cor} \label{C:genfun}
For every invariant generalized function $\chi\in \wh{C}^{G}(G)$ of depth $\leq r$ and every $h\in\C{H}(G)$, we have the equality
\[
\chi(h)=\lim_{\Si\in\T_m} \left[\sum_{\si\in\Si}(-1)^{\dim\si} \chi(\dt_{G_{\si,r^+}}\ast h\ast\dt_{G_{\si,r^+}})\right].
\]
\end{Cor}

\begin{proof}
Since $\chi$ is of depth $\leq r$, we have the equality  $\chi(h)=(\Pi_r(\chi))(h)=\chi(\Pi_r(h))$. Then by
\rt{explicit}, $\chi(h)$ equals
\[\lim_{\Si\in\T_m}
\chi(E^{\Si}_r\ast h)=\lim_{\Si\in\T_m}\left[\sum_{\si\in\Si}(-1)^{\dim\si}\chi(\dt_{G_{\si,r^+}}\ast h)\right].\]
Finally, since $\chi$ is $\Ad G$-invariant, we have
\[\chi(\dt_{G_{\si,r^+}}\ast h\ast\dt_{G_{\si,r^+}})= \chi(\dt_{G_{\si,r^+}}\ast\dt_{G_{\si,r^+}}\ast h)=\chi(\dt_{G_{\si,r^+}}\ast h),\]
and the assertion follows.
\end{proof}

\begin{Emp}
\begin{proof}[Proof of \rco{Av}]
Note that $\Av_{Y_{\P}^{\Si}}(E\ast \dt_{\P^+_n})=E\ast \Av_{Y_{\P}^{\Si}}(\dt_{\P^+_n})$, since $E$ is $\Ad G$-invariant, and that
$E_n^{\Si}=\sum_{\P\in\Par}(-1)^{\dim \si_{\P}}\Av_{Y_{\P}^{\Si}}(\dt_{\P^+_n})$ for every  $\Si\in\T$.
Thus the right-hand side of \form{aver} equals $\lim_{\Si\in\T}(E\ast E_n^{\Si})$.

Next, for every $f\in C_c^{\infty}(G)$ we have $\Pi_n(E)(f)=E(\Pi_n(f))=\lim_{\Si\in\T}E(\Si_n^{\Si}\ast f)$.
Thus it remains to show that for every $\Si\in\T$ we have $E(E_n^{\Si}\ast f)= (E\ast E_n^{\Si})(f)$.

Since $(E\ast E_n^{\Si})(f)=E(f\ast \iota^*(E_n^{\Si}))$, where $\iota$ is as in \re{pfexplicit}, and $\iota^*(E_n^{\Si})=E_n^{\Si}$,
we are reduced to the equality
$E(E_n^{\Si}\ast f)=E(f\ast E_n^{\Si})$, which holds because $E$ is $\Ad G$-invariant.
\end{proof}
\end{Emp}

\begin{Emp} \label{E:pflie}
\begin{proof}[Proof of \rp{lie}]
For every $\Si\in \T_m$ we set $\g^*_{\Si,-r}:=\cup_{\si\in\Si}\g^*_{\si,-r}$.
Then $\g^*_{\Si,-r}\subseteq\g^*$ is an open and compact subset, and $\g^*_{-r}=\cup_{\Si\in \T_m}\g^*_{\Si,-r}$.
Thus we have  $1_{\g^*_{-r}}= \lim_{\Si\in\T_m} 1_{\g^*_{\Si,-r}}$, hence
$\C{E}_r=\lim_{\Si\in\T_m}\C{F}^{-1}(1_{\g^*_{\Si,-r}})$. It therefore suffices to show that $\C{F}^{-1}(1_{\g^*_{\Si,-r}})=\C{E}^{\Si}_r$,
that is, $\C{F}(\C{E}^{\Si}_r)=1_{\g^*_{\Si,-r}}$.

Notice that the restriction of the Fourier transform $\C{F}:D(\g)\to\wh{C}(\g^*)$ to
$\C{H}(\g)$ is the Fourier transform $\C{H}(\g)\to C_c^{\infty}(\g^*)$.

Since $\psi$ is trivial on $(\varpi)$ but nontrivial on $\C{O}$,
for every $\si\in [\CX_m]$ the lattice $\g^*_{\si,-r}\subseteq\g^*$ is the orthogonal complement
of $\g_{\si,r^+}\subseteq\g$ with respect to the pairing $\g\times\g^*\to\B{C}\m:(a,b)\mapsto\psi(\lan b,a\ran)$.
Thus, we have the equality $\C{F}(\dt_{\g_{\Si,r^+}})=1_{\g^*_{\Si,-r}}$, hence
$\C{F}(\C{E}^{\Si}_r)=\sum_{\si\in\Si}(-1)^{\dim\si}
1_{\g^*_{\si,-r}}$. Therefore it suffices to show the following result.
\end{proof}
\end{Emp}

\begin{Lem} \label{L:ep}
For every $\Si\in\T_m$, we have the equality $1_{\g^*_{\Si,-r}}=\sum_{\si\in\Si}(-1)^{\dim\si}
1_{\g^*_{\si,-r}}$.
\end{Lem}

\begin{proof}
Set  $\varphi_{\Si}:=\sum_{\si\in\Si}(-1)^{\dim\si}
1_{\g^*_{\si,-r}}$. Clearly,
$\varphi_{\Si}(b)=0$ if $b\in\g^*\sm\g^*_{\Si,-r}$, so it remains to show that
$\varphi_{\Si}(b)=1$ if $b\in\g^*_{\Si,-r}$.

For every $b\in \g^*$, we denote by $[\CX_m](b)$ the set of $\si\in [\CX_m]$ such that $b\in \g^*_{\si,-r}$, and set
$\Si(b):=[\CX_m](b)\cap \Si$. Since $\g^*_{\si,-r}\subseteq \g^*_{\si',-r}$ for every $\si'\preceq\si$, we conclude that
$[\CX_m](b)$ and hence also $\Si(b)$ is a subcomplex of $[\CX_m]$.

By the definition of $\varphi_{\Si}$,  the value
$\varphi_{\Si}(b)$ equals the Euler--Poincar\'e characteristic  of $\Si(b)$. Thus it suffices to show
that for every $b\in \g^*_{\si,-r}$ the complex $\Si(b)$ is convex.
The complex $\Si\in\T_m$ is convex by assumption, hence it remains to show that the complex $[\CX_m](b)$ is convex. Since $|[\CX_m](b)|$ is the convex set $\C{X}(b,r)$ from \rl{convex}, we are done.
\end{proof}

\begin{Emp}
\begin{proof}[Proof of \rco{lie}]
Since $\C{L}$ induces a homeomorphism $G_{\si,r^+}\isom\g_{\si,r^+}$ for every $\si\in [\CX_m]$,
it satisfies $\C{L}_!(\dt_{G_{\si,r^+}}|_{G_{r^+}})=\dt_{\g_{\si,r^+}}|_{\g_{r^+}}$. Hence $\C{L}_!(E^{\Si}_r|_{G_{r^+}})=\C{E}_r^{\Si}|_{\g_{r^+}}$ for every $\Si\in \T_m$. We conclude by
\rco{explicit} and \rp{lie}.
\end{proof}
\end{Emp}

\section{Relation to the character of the Steinberg representation}
\noindent In this section we prove \rt{steinberg}.

\begin{Emp} \label{E:stfin}
{\bf Steinberg representations of finite groups} (compare \cite[3.2,4.2]{Cu}).
For an algebraic group $\L$ over a finite field $\fq$, we set $L:=\L(\fq)$.

(a) Let $\L$ be a connected reductive group over a finite field $\fq$, $\b\subseteq \L$ a Borel subgroup, and $\U\subseteq \b$ the unipotent radical of $\b$. Then the Hecke algebra $\C{H}(L,B)$ has a basis $h_w:=\frac{1}{|B|}1_{BwB}$, parameterized by elements $w$ of the Weyl group $W_L$ of $\L$, where $|B|$ denotes the cardinality of $B$.

(b) Let $\St_L$ be the Steinberg representation of $L$. Then $\St_L$ is an irreducible representation, the space of invariants  $\St_L^B$ is a one-dimensional representation of the Hecke algebra $\C{H}(L,B)$, and each $h_w$ acts on $\St_L^B$ as $\sgn(w)\Id$.

(c) The restriction of $\St_L$ to $U=\U(\fq)$ is the regular representation. Therefore $\Tr(1,\St_L)=|U|$, and
$\Tr(g,\St_L)=0$ for every unipotent element
$1\neq g\in L$.
\end{Emp}


\begin{Emp} \label{E:stpad}
{\bf Steinberg representations of $p$-adic groups} (see \cite{Bo}, or \cite[Section 8]{Ca} and \cite[p. 199-205]{BW}).

(a) Let $\St_G$ be the Steinberg (or special) representation of $G=\G(F)$. Then $\St_G$ is irreducible, the space of Iwahori invariants $\St_G^I$ is a one-dimensional module of the Hecke algebra $\C{H}(G,I)$, and for every element $w$ of the affine Weyl group $W_G^{\aff}$ of $G$,
the element $1_{IwI}\dt_I\in \C{H}(G,I)$ acts on $\St_G^I$ as $\sgn(w)\Id$.

(b) As a virtual representation, $\St_G$ equals the alternating sum of
the non-normalized induced representations $\Ind_{Q}^{G}(1_{Q})$, where $Q=\Q(F)$, and
$\Q$ runs over the set of standard parabolic subgroups $\Q\subseteq \G$.
\end{Emp}

\begin{Emp} \label{E:iwpar}
{\bf Parahoric subgroups.}
(a) Fix a parahoric subgroup $P\subseteq G$ and an Iwahori subgroup $I\subseteq P$.
Then the quotient $P/P^+$ is naturally isomorphic to $L=\L(\fq)$ for
some connected reductive group $\L=\L_P$ over $\fq$. Under this isomorphism $I/P^{+}\subseteq P/P^{+}$
corresponds to $B=\b(\fq)$ for some Borel subgroup $\b=\b_P\subseteq\L$.

(b) Note that for every representation $V\in R(G)$, the space of invariants $V^{P^+}$ is a representation of  $P/P^+=L$.
\end{Emp}

\begin{Prop} \label{P:restr}
The $L$-representation $\St_G^{P^+}$ is isomorphic to the Steinberg representation $\St_{L}$.
\end{Prop}

\begin{proof}
Denote the $L$-representation $\St_G^{P^+}$ by $\St'$. Then  $(\St')^{B}=\St_G^I$, and we have natural embeddings $W_{L}\hra\wt{W}$ and  $\C{H}(L,B)\hra \C{H}(G,I)$ under which $h_w$ from \re{stfin}(a)  corresponds to $1_{IwI}\dt_I\in \C{H}(G,I)$. Therefore, by \re{stpad}(a), $(\St')^{B}$ is a one-dimensional representation of the Hecke algebra  $\C{H}(L,B)$ such that  $h_w$ acts on it as $\sgn(w)\Id$ for every $w\in W_{L}$. Hence, by \re{stfin}(b), $\St'$ is isomorphic to a direct sum $\St_{L}\oplus V$ with $V^{B}=0$. It remains to show that $\St'$ is generated by its $B$-invariants. But this follows
from \rl{inv} below.
\end{proof}

\begin{Lem} \label{L:inv}
For every smooth representation $V$ of $G$, which is generated by its $I$-invariants,
the $L$-representation $V^{P^+}$ is generated by $B$-invariants.
\end{Lem}

 \begin{proof}
 Since $V$ is generated by $V^I$, it is  a quotient of a direct sum of the $C_c^{\infty}(I\bs G)$'s. Thus,
it is enough to prove the assertion in the case $V=C_c^{\infty}(I\bs G)$.
In this case the space $V$, considered as a $P$-representation, decomposes as a sum
$V=\sum_{g\in G}V_g$, where $V_g:=\B{C}[I\bs IgP]$. Thus it remains to show that each $V_g^{P^+}$ is generated by its
$B$-invariants. It suffices to show that $V_g^{P^+}\cong \B{C}[B'\bs L]$, where $B'=\b'(\fq)$
for some Borel subgroup $\b'\subseteq\L$.

Notice that we have a natural isomorphism of $P$-representations $V_g\cong\B{C}[P\cap I'\bs P]$, where $I':=g^{-1}Ig$. Therefore
$V_g^{P^+}\cong\B{C}[P^+(P\cap I')\bs P]$,  so it suffices to show that $J:=P^+(P\cap I')\subseteq P$
is an Iwahori subgroup (compare \re{iwpar}).

\begin{Emp} \label{E:notinv}
{\bf Notation.} For every $\si\in[\CX]$, choose $x\in\si$ and define $G_{\si}:=G_{x,0}$ (use \rl{mpfil}).
\end{Emp}

Let $\si\in [\CX]$ (resp. $\tau\in [\CX]$) be the polysimplex such that $P=G_{\si}$ (resp. $I'=G_{\tau}$).
Choose an apartment $\CA\supseteq\si,\tau$ of $\CX$ and points $x\in \si$ and
$y\in \tau$. Since $I'$ is an Iwahori subgroup, $\tau$ is a chamber. Hence  we have $\psi(y)\neq 0$ for every $\psi\in\Psi(\CA)$.
Therefore every point $z\in (x,y]$, close to $x$, lies in some chamber $\wt{\si}\in[\CA]$ such that $\si\preceq\wt{\si}$.
We claim that $J=G_{\wt{\si}}$, that is, $G_{\wt{\si}}=G_{\si,0^+}(G_{\si}\cap G_{\tau})$.

By \rp{mp}, the subgroup $G_{\wt{\si}}$ is generated by $M_0$ and the affine root subgroups $U_{\psi}$ for $\psi\in\Psi(\CA)$ satisfying $\psi(\wt{\si})>0$. Since  $\si\preceq\wt{\si}$, we have $\psi(\wt{\si})>0$ if and only if we have either
$\psi(\si)>0$ or $\psi(\si)=0$ and $\psi(\wt{\si})>0$. Thus, to show the inclusion
$G_{\wt{\si}}\subseteq G_{\si,0^+}(G_{\si}\cap G_{\tau})$, we have to check that for every $\psi\in\Psi(\CA)$, satisfying $\psi(\si)=0$ and $\psi(\wt{\si})>0$, we have $\psi(\tau)>0$. Equivalently, we have to check that for every $\psi\in\Psi(\CA)$,
satisfying $\psi(x)=0$ and $\psi(z)>0$ we have $\psi(y)>0$, which follows from the assumption $z\in (x,y]$.

The converse inclusion is easier. Namely, the inclusion $G_{\si,0^+}\subseteq G_{\wt{\si},0^+}\subseteq G_{\wt{\si}}$ follows from the fact that $\si\preceq\wt{\si}$, while the inclusion $G_{\si}\cap G_{\tau}\subseteq G_{\wt{\si}}$ or, equivalently,
$G_{x}\cap G_{y}\subseteq G_z$ follows from \rl{convex}.
\end{proof}

To prove \rt{steinberg}, we are going to use a result of  Meyer--Solleveld \cite[Prop. 4.1]{MS}, which we are going to formulate now.

\begin{Emp} \label{E:MS}
{\bf Theorem of Meyer--Solleveld} (see \cite[Section 4]{MS}).

(a) For every $\si\in [\CX]$, we denote by $G_{\si}^{\dag}\subseteq G$ the stabilizer of $\si$, and let $\sgn_{\si}:G_{\si}^{\dag}\to\{\pm 1\}$ be the orientation character, that is, $\sgn_{\si}(g)=1$ if and only if $g\in G_{\si}^{\dag}$ preserves an orientation of $\si$. In particular, the restriction $\sgn_{\si}|_{G_{\si}}$ is trivial.

(b) Let $n\in\B{N}$, let  $V\in R(G)_{\leq n}$ be a finitely generated admissible representation, and let $\chi_V\in\wh{C}^G(G)$ be its character.
Since $G_{\si}^{\dag}$ normalizes $G_{\si,n^+}$, it acts on the space of invariants $V^{G_{\si,n^+}}$.

(c) A result of Meyer--Solleveld \cite[Prop 4.1]{MS} asserts that for every compact open subgroup $K\subseteq G$, function $f\in C_c^{\infty}(K)$, Haar measure $dg$ on $G$ and sufficiently large $K$-invariant finite subcomplex $\Sigma\in\T$, we have the equality
\begin{equation} \label{Eq:ss}
\chi_{V}(fdg)=\int_{g\in K} f(g)\left(\sum_{\si\in\Sigma\,|g\in G_{\si}^{\dag}}(-1)^{\dim\si}\sgn_{\si}(g)\Tr(g,V^{G_{\si,n^+}})\right)dg.
\end{equation}
\end{Emp}

\begin{Emp} \label{E:comparison}
{\bf Remark.} Note that there is a lot of similarity between formula \form{ss} of Meyer--Solleveld and our \rt{explicit}. However
we don't know whether one of these results formally implies the other (compare also remark \re{comp}).
\end{Emp}

\begin{Emp}
\begin{proof}[Proof of \rt{steinberg}]
We have to show that the equality
\begin{equation} \label{Eq:eqst}
E_0(f)=\chi_{\St_G}(f\mu^{\I^+})
\end{equation}
is valid for every $f\in C_c^{\infty}(G_{0^+})$.
Moreover, since $E_0$ and $\chi_{\St_G}$ are $\Ad G$-invariant and
$G_{0^+}=(\Ad G)(I^+)$,  it is enough to prove \form{eqst} for $f\in C^{\infty}_c(I^+)$.

To calculate the right-hand side of \form{eqst}, we apply formula \form{ss} for $n=0$, $V=\St_G$, $K=I^+$ and $dg=\mu^{I^+}$.
We set $G_{\si}^+:=G_{\si,0^+}$, $L_{\si}:=L_{G_{\si}}$, and let $U_{\si}\subseteq L_{\si}$ be a maximal unipotent subgroup.

Notice that for every $g\in I^+\cap G_{\si}^{\dag}$, we have $g\in G_{\si}$, and the image $[g]\in L_{\si}$ is unipotent. In particular, $\sgn_{\si}(g)=1$. Since the space of invariants $\St_G^{G_{\si}^+}$ is the Steinberg representation  of $L_{\si}$ (by \rp{restr}),
we conclude from \re{stfin}(c) that for every $g\in I^+\cap G_{\si}^{\dag}$,
the trace $\Tr(g,\St_G^{G_{\si}^+})$ equals $|U_{\si}|1_{G_{\si}^+}(g)$.

Hence, by \form{ss}, the right-hand side of \form{eqst} equals
\begin{equation} \label{Eq:ss1}
\int_{g\in I^+} f(g)\left(\sum_{\si\in\Sigma}(-1)^{\dim\si}|U_{\si}|1_{G^+_{\si}}(g)\right)\mu^{I^+}
\end{equation}
for every sufficiently large $I^+$-invariant subcomplex $\Sigma\in\T$. Using the identity $|U_{\si}|1_{G^+_{\si}}\mu^{I^+}=\dt_{G^+_{\si}}$, the expression \form{ss1} equals
$\int_{g\in I^+}f(g)E_0^{\Si}=E_0^{\Si}(f)$. This implies that $\chi_{\St_G}(f\mu^{\I^+})=E_0^{\Si}(f)$,
hence equality \form{eqst} follows from \rco{explicit}(a).
\end{proof}
\end{Emp}

\section{Stability}

\noindent In this section we prove \rco{stable} and \rt{stable}.

\begin{Emp} \label{E:setupst}
{\bf Set up.} (a) We fix a non-zero translation invariant top degree differential form $\om_\G$ on $\G$ and such a form $\om_\TT$ on $\TT$ for each maximal torus $\TT\subseteq\G$. Then $\om_\G/\om_\TT$ is a top degree translation invariant differential form on $\G/\TT$, hence it defines a $G$-invariant measure $|\om_\G/\om_\TT|$ on $(\G/\TT)(F)$. Also $\om_\G$ defines a Haar measure $|\om_\G|$ on $G$.

(b) Let $\X$ be either $\G$, or $\CG$, or $\CG^*$, where $\CG$ denotes the Lie algebra $\g$ viewed as an algebraic variety, and similarly for $\CG^*$. Then $\X$ is equipped with an adjoint action of $\G$. We denote by
$\X^{\sr}\subseteq \X$ the set of {\em strongly regular semisimple} elements of $\X$, that is, the set of all $x\in \X$ such that the stabilizer
$\G_x:=\Stab_\G(x)\subseteq \G$ is a maximal torus. Then $\X^{\sr}\subseteq\X$ is an open subvariety.

(c) We assume that $\X^{\sr}\neq\emptyset$. Note that this is always holds, if
$\X=\G$ or the characteristic of $F$ is not two (see \re{weyl} below and compare \cite[Prop. 2.3]{GG}).

(d) We set $X:=\X(F)$ and $X^{\sr}:=\X^{\sr}(F)$. Then the subset $X^{\sr}\subseteq X$ is dense.
\end{Emp}

\begin{Emp} \label{E:weyl}
{\bf Remark.} Let $\TT\subseteq\G$ be a maximal torus, set $\t:=\Lie\TT$, and let $\t^*$ be the linear dual of $\t$. Then it is standard that $\G^{\sr}\neq\emptyset$ (resp. $\g^{\sr}\neq\emptyset$, resp. $(\g^*)^{\sr}\neq\emptyset$) if and only if the Weyl group
$W=W(\G,\TT)$ acts faithfully on $\TT$ (resp. $\t$, resp. $\t^*$). Then the assertions for $\TT$ and in the  characteristic zero case follow from the fact that $W(\G,\TT)$ acts faithfully on $X_*(\TT)$.

On the other hand, in characteristic $p>0$ the assertion for  $\t$ (resp. $\t^*$) is equivalent to the assertion
that $W$ acts faithfully on $X_*(\TT)/pX_*(\TT)$ (resp. $X^*(\TT)/pX^*(\TT)$).

We claim that both assertions hold if $p>2$. Indeed, let $w\in W$ acts trivially on $X_*(\TT)/pX_*(\TT)$. Equip the vector space $V:=X_*(\TT)\otimes\B{R}$ with a $W$-invariant norm $||\cdot||$. Then the endomorphism $A:=\frac{w-1}{p}\in \End(V)$ satisfies $A(X_*(\TT))\subseteq X_*(\TT)$ and $||A(v)||\leq \frac{2}{p}||v||<||v||$ for every $v\in V$. Since $A$ is semisimple, we conclude that $A=0$, hence $w=1$. The proof of the assertion for $X^*(T)$ is identical.
\end{Emp}

\begin{Emp} \label{E:stable}
{\bf Stability.} Suppose that we are in the situation of \re{setupst}.

(a) For every $x\in X^{\sr}$ we have a natural map $a_x:\G/\G_x\to \X:[g]\mapsto g(x)$,
hence a map $(\G/\G_x)(F)\to X^{\sr}$, whose image we call the {\em stable orbit}.

(b) Notice that each stable orbit is closed in $X$, hence we can define an invariant distribution $O^{\st}_x\in D^G(X)$ by the formula
$O^{\st}_x(f):=\int_{(\G/\G_x)(F)}a_x^*(f)|\om_\G/\om_{\G_x}|$ for every smooth function with compact support
$f\in C_c^{\infty}(X)$. The distribution $O^{\st}_x$ is called the {\em stable orbital integral}. It is defined uniquely up to a constant.

(c) A function $f\in C_c^{\infty}(X)$ is called {\em unstable}, if  $O^{\st}_x(f)=0$ for every $x\in X^{\sr}$.
An invariant distribution $F\in D^G(X)$ is called {\em stable}, if $F(f)=0$ for every unstable $f\in  C_c^{\infty}(X)$.
An invariant generalized function $\chi\in\wh{C}^G(X)$ is called {\em stable}, if $\chi dx\in D^G(X)$ is stable  for a Haar measure
$dx$ on $X$.

(d) We call an $\Ad G$-equivariant open and closed subset $Y\subseteq X$ {\em stable}, if $Y\cap X^{\sr}$ is a union of stable orbits
(see (a)).
\end{Emp}

\begin{Emp} \label{E:stsubs}
{\bf Examples.} (a) If $Y\subseteq X$ is a stable subset (see \re{stable}(d)), then the characteristic function $1_Y\in \wh{C}^G(X)$
is stable.

Indeed, we want to show that for every unstable function $f\in C_c^{\infty}(G)$ we have
$\int_{G}(f\cdot 1_Y)dx=0$. Since $Y$ is stable, the function
$f\cdot 1_Y\in C_c^{\infty}(X)$ is unstable. Thus it remains to check that for every unstable
function $f\in C_c^{\infty}(X)$ we have $\int_{G}fdx=0$. This follows from
the fact $X^{\sr}\subseteq X$ is dense.

(b) The character $\chi_{\St_G}$ of the Steinberg representation is stable.

Indeed, by \re{stpad}(b) it remains to show
that each character $\chi_{\Ind_{Q}^{G}(1_{Q})}$ is stable. This follows from the fact that
the constant function $1_Q$ is stable (by (a)) and that the parabolic induction preserves stability (see   \cite[Cor 6.13]{KV3}).
\end{Emp}

The following lemma will be proven in Appendix \ref{S:cong} (see \re{disc}(b)).

\begin{Lem} \label{L:clos}
For every $r\in\B{R}_{\geq 0}$, the open $\Ad G$-invariant subsets $G_{r^+}\subseteq G$, $\g_{r^+}\subseteq \g$ and
$\g^*_{-r}\subseteq \g$ are closed and stable.
\end{Lem}

\begin{Emp}
{\bf Remark.}
The fact that $G_{r^+}\subseteq G$ and  $\g_{r^+}\subseteq \g$ are closed was also proven by Adler and DeBacker (see \cite[Cor 3.4.3 and Cor 3.7.21]{ADB}).
Our proof is completely different.
\end{Emp}

\begin{Emp} \label{E:pfcstable}
\begin{proof}[Proof of \rco{stable}]
We have to show that for every unstable $f\in C_c^{\infty}(G)$, we have $E_0(f)=0$.

Since $G_{0^+}\subseteq G$ is open and closed
(by \rl{clos}), $f$ decomposes as $f=f'+f''$, where $f':=f\cdot 1_{G_{0^+}}$
and $f'':=f\cdot 1_{G\sm G_{0^+}}$. Since $f$ is unstable, while $G_{0^+}\subseteq G$ is stable (by \rl{clos}),
we conclude that $f'$ is unstable.

Since $E_0$ is supported on $G_{0^+}$ (by \rco{explicit}), and  $f''$ is supported on $G\sm G_{0^+}$,
we conclude that $E_0(f'')=0$. Therefore $E_0(f)=E_0(f')$ equals $\chi_{\St_G}(f'\mu^{I^+})$
(by \rt{steinberg}). Hence $E_0(f)=\chi_{\St_G}(f'\mu^{I^+})=0$, because $\chi_{\St_G}$ is stable
(see \re{stsubs}(b)), while $f'$ is unstable.
\end{proof}
\end{Emp}

\begin{Cor} \label{C:wald}
Assume that the characteristic of $F$ is different from two, and that $\G$ admits an $r$-logarithm. Then the invariant distribution $E_r$ is stable.
\end{Cor}
\begin{proof}
By Example \re{stsubs}(a) and \rl{clos}, the invariant generalized function  $1_{\g^*_{-r}}\in\wh{C}^G(\g^*)$ is stable. Hence, by
a generalization \cite{KP} of a theorem of Waldspurger \cite{Wa}, the distribution $\C{E}_r=\C{F}^{-1}(1_{\g^*_{-r}})$ is stable.

The rest of the argument is similar to \re{pfcstable}. For every unstable function  $f\in C_c^{\infty}(G)$, functions
$f':=f\cdot 1_{G_{r^+}}$ and $\C{L}_!(f')\in C_c^{\infty}(\g)$ are unstable. On the other hand, we have
$E_r(f)=E_r(f')$, because $E_r$ supported on $G_{r^+}$, and $E_r(f')=\C{E}_r(\C{L}_!(f'))$ by \rco{lie}. Hence
$E_r(f)=\C{E}_r(\C{L}_!(f'))=0$, because $\C{E}_r$ is stable.
\end{proof}

\begin{Emp} \label{E:remwald}
{\bf Remarks.}
(a) Formally speaking, the theorem of Waldspurger and its generalization in \cite{KP} are only proved when $F$ is of characteristic zero.
But the arguments can be extended to local fields of positive odd characteristic.

(b) In all known cases when $G$ admits an $r$-logarithm, the Lie algebra admits a non-degenerate quadratic form. In this case,
we can identify $\g^*$ with $\g$, thus the original theorem of Waldspurger suffices.

(c) When $r\in\B{N}$, we can prove \rco{wald} without the theorem of Waldspurger. Namely, arguing as in the second paragraph of the proof of \rco{wald}, we see that $E_r$ is stable if and only if $\C{E}_r$ is stable. Hence, by \rco{stable},  it suffices to show that
$\C{E}_r$ is stable if and only if $\C{E}_0$ is stable.

Let
$\mu_r:\g\to\g$ be the homothety map $a\mapsto \varpi^r a$. Since $r\in\B{N}$, for every $x\in\CX$, we have the equality
$\g_{x,r^+}=\varpi^r \g_{x,0^+}$ (see \re{lattice}(b)). Then the pullback $\mu_r^*:D(\g)\to D(\g)$ satisfies
$\mu_r^*(\dt_{\g_{x,r^+}})= \dt_{\g_{x,0^+}}$ for all $x\in \CX$, hence
$\mu_r^*(\C{E}^{\Si}_{r})=\C{E}_{0}^{\Si}$ for all $\Si\in\T_m$. Thus $\mu_r^*(\C{E}_{r})=\C{E}_{0}$ by \rp{lie}.
Since $\mu_r^*$ maps stable distributions to stable distributions, the assertion follows.
\end{Emp}

\begin{Emp} \label{E:verygood}
{\bf (Very) good primes.} (a) Let $\G^{\sc}$ be the simply connected covering of the derived group of $\G$. Then $\G^{\sc}$ decomposes as a product $\G^{\sc}=\prod_i \R_{F_i/F}\H_i$, where each $F_i/F$
is a finite separable extension, $\H_i$ is an absolutely simple algebraic group over $F_i$,
and $\R_{F_i/F}$ denotes the Weil restriction of scalars. We denote by $\H^*_i$ the quasi-split
inner form of $\H_i$ and by $F_i[\H_i^*]$ the splitting field of $\H_i^*$.

(b) We say that $p$ is {\em good} for $\G$, if either $p>5$, or

$\bullet$ $p=5$ and none of the $\H_i$'s is of type $E_8$, or

$\bullet$ $p=3$, each of the $\H_i$'s is of types $A-D$ and satisfies $[F_i[\H_i^*]:F_i]\leq 2$.

(c) We say that $p$ is {\em very good} for $\G$, if $p$ is good, and $p$ does
not divide $n$, if some of the $\H_i$'s is of type $A_n$.
\end{Emp}

The following assertion is an immediate consequence of Lemmas \ref{L:qlogo} and \ref{L:qlog} from Appendix
\ref{S:qlog}.

\begin{Cor} \label{C:rlog}
If $p$ is very good for $\G$, then $\G^{\sc}$ admits an $r$-logarithm for every $r\in\B{R}_{\geq 0}$.
\end{Cor}

The proof of following assertion is given in Appendix \ref{S:cong} (see \re{pfisog}).

\begin{Lem} \label{L:isog}
Let $\pi:\G'\to \G$ be an isogeny of degree prime to $p$. Then $\pi$ induces homeomorphisms
$G'_{x,r^+}\isom G_{x,r^+}$ and $G'_{r^+}\isom G_{r^+}$ for all $r$ and $x\in\CX(\G')=\CX(\G)$.
\end{Lem}

\begin{Cor} \label{C:isog}
In the situation of \rl{isog}, the distribution $E_r$ on $G$ is stable if and only if $E_r$ on $G'$ is stable.
\end{Cor}

\begin{proof}
Since $E_r$ is supported on $G_{r^+}$ (by \rco{explicit}(b)), to show that it is stable, we have to check that
$E_r(f)=0$ for every unstable $f$ supported on $G_{r^+}$, and similarly for $\G'$.
Thus, the assertion follows from \rl{isog} and \rco{explicit}(a).
\end{proof}

\begin{Emp} \label{E:pfstable}
\begin{proof}[Proof of \rt{stable}]
Consider the natural isogeny $\pi:\G^{\sc}\times \Z(\G)^0\to \G$. Since the degree of $\pi$ divides $|\Z(\G^{\sc})|$, and $p$ is very good,
the degree of $\pi$ is prime to $p$. Hence, by \rco{isog}, to show the stability of $E_r$ on $G$, it is enough to show the stability of $E_r$ on $G^{\sc}$. Since $\G^{\sc}$ admits an $r$-logarithm by \rco{rlog}, the assertion follows from \rco{wald}.
\end{proof}
\end{Emp}

\begin{Emp} \label{E:good}
{\bf Remark.} If $F$ is of characteristic zero and $p$ is good, then $E_r$ is stable. Indeed, arguing similarly to
\re{pfstable}, we reduce to the assertion that $E_r$ is stable, if each $\H_i$ if of type $A$ and $p>2$. Then, using classification and the assumption that the characteristic of $F$ is zero,
we reduce to the case when $\G$ is either $\GL_n$ of $\GU{n}$.
In both cases, $\G$ admits an $r$-logarithm, so the assertion follows from \rco{wald}.
\end{Emp}

\appendix

\section{Properties of Moy--Prasad filtrations} \label{S:mp}

\noindent In this section we provide proofs of some of the results, formulated in Sections \ref{S:bt} and \ref{S:mpf}. We are going to follow a standard strategy, first to pass to an unramified extension, thus reducing to a quasi-split case, then to pass to a Levi subgroup, thus reducing to a rank one case, and
to finish by direct calculations. Though most of the results in this sections are well-known to specialists (see, for example, \cite[Section 1]{Vi}), we include details for completeness.

\begin{Emp} \label{E:setup}
{\bf Set-up.} Let $\S\subseteq \G$ be a maximal split torus, $\M:=\Z_\G(\S)$ the corresponding minimal Levi subgroup of $\G$, set
$\C{A}:=\C{A}_{\S}$, and let $\Phi(\C{A})_{nd}\subseteq \Phi(\C{A})$ be the set of non-divisible roots, that is, those
$\al\in \Phi(\CA)$ such that $a/2\notin\Phi(\C{A})$.
\end{Emp}

\begin{Lem} \label{L:f'}
There exists a finite unramified extension $F'/F$ such that $\G':=\G_{F'}$ is quasi-split. Moreover, for every such extension, there exists a subtorus $\S'\supseteq\S$ of $\G$ defined over $F$ such that $\S'_{F'}\subseteq \G'$ is a maximal split torus.
\end{Lem}

\begin{proof}
Assume first that  $\G=\GL_{1}(D)$ for some finite-dimensional central division algebra $D$ over $F$. In this case, both assertions are easy.
Indeed, let $\dim_F D=d^2$, and let $F'/F$ be an unramified extension. Then $\G_{F'}$ is quasi-split if and only if $F'$ splits $D$. Moreover,  this happens if and only if $F'\supseteq F^{(d)}$, where $F^{(d)}/F$ is an unramified extension of degree $d$.
Furthermore, there exists an embedding $F^{(d)}\hra D$ of $F$-algebras, whose image corresponds to a torus $\S'$ we are looking for.

Assume next that $\G=\GL_{1}(D)$ for some (not necessary central) finite-dimensional division algebra $D$ over $F$. This case reduces to the first one,
and is left to the reader.

Finally, the general case follows from the previous one. Indeed, $\G_{F'}$ is quasi-split if and only if $\M_{F'}$ is quasi-split, and if and only if the simply connected covering $\M_{F'}^{\sc}$ of $\M_{F'}$ is quasi-split. Thus we may replace $\G$ by $\M^{\sc}$, thus assuming that $\G$ is semisimple, simply-connected, and anisotropic. Next, decomposing $\G$ into simple factors, we may further assume that
$\G$ is simple. Then $\G=\SL{1}(D)$ for some finite-dimensional division algebra over $F$, and $\SL{1}$ denotes the kernel of the reduced norm (see \cite[Thm 6.5. p. 285]{PlR}). Since the assertion for $\SL{1}(D)$ follows from the assertion for $\GL_{1}(D)$, the proof is now complete.
\end{proof}

\begin{Emp} \label{E:order}
{\bf Affine roots subgroups.}  (a) Choose a set of positive roots $\Phi(\C{A})^+_{nd}\subseteq \Phi(\C{A})_{nd}$, and a total order on $\Phi(\C{A})_{nd}\cup\{0\}$ such that $\al>0$ if and only if $\al\in\Phi(\C{A})^+_{nd}$.
Set $\U_0:=\M$. Then the product map $\prod_{\al\in \Phi(\C{A})_{nd}\cup\{0\}}\U_{\al}\to \G$ is an open embedding.

(b) For every $\al\in\Phi(\CA)$, $x\in\CA$ and $r\in\B{R}_{\geq 0}$, we denote by $\psi_{\al,x,r}$ the smallest affine root $\psi\in\Psi(\C{A})$ such that $\al_{\psi}=\al$ and $\psi(x)\geq r$. Set $U_{\al,x,r}:=U_{\psi_{\al,x,r}}\subseteq U_{\al}$ and
$\u_{\al,x,r}:=\u_{\psi_{\al,x,r}}\subseteq \u_{\al}$.

(c) We also set $U_{(\al),x,r}:=U_{\al,x,r}\cdot U_{2\al,x,r}\subseteq U_{\al}$, if  $2\al\in \Phi(\CA)$; $U_{(\al),x,r}:=U_{\al,x,r}$, if  $2\al\notin \Phi(\CA)$; and $U_{0,x,r}:=M_r$.
\end{Emp}

\begin{Emp} \label{E:sl2}
{\bf The $\SL{2}$-case.} Let $\G=\SL{2}$, and let $\S\subseteq\G$ be the group of diagonal matrices. In this case, $\G$ and $\S$ have natural $\C{O}$-structures, hence
 we have a natural identification $\C{A}\isom V_{\G,\S}$ (see \re{apart}(a)), which identifies  $\Phi(\C{A})$ with $\pm \al$ and $\Psi(\C{A})$ with $\pm\al+\B{Z}$. Moreover, if the root subgroup $U_{\al}$ consists of matrices  $g_a=\left(\begin{matrix}
1 & a \\
0& 1\end{matrix}\right)$
with $a\in F$, then the affine root subgroup $U_{\al+n}\subseteq U_{\al}$ consists of $g_a\in U_{\al}$ with $\val_F(a)\geq n$.
\end{Emp}

\begin{Emp} \label{E:su3}
{\bf The $\SU{3}$-case} (compare \cite[Ex. 1.15]{Ti}). (a) Let $K/F$ be a separable totally ramified quadratic extension, and let $\tau\in\Gal(K/F)$ be a non-trivial element.
Let $\G=\SU{3}$ be the special unitary group over $F$ split over $K$, corresponding to the quadratic form
$(\ov{x},\ov{y})\mapsto\sum_i x_i y^{\tau}_{3-i}$. Let $\S\subseteq\G$ the maximal torus, corresponding to diagonal matrices, and let
$\al\in\Phi(\G,\S)$ be the non-divisible root such that $U_{\al}$ consists of upper triangular matrices.
Then $U_{\al}$ consists of all elements of the form
$g_{a,b}=
\left(\begin{matrix}
1 & -a & -b\\
0& 1 & a^{\tau}\\
0& 0& 1
\end{matrix}\right), a,b\in K \text{ such that }a a^{\tau}+b+b^{\tau}=0,
$
while $U_{2\al}$ consists of all $g_{0,b}\in U_{\al}$.

(b) Set $\dt:=\max\{\val_K(b)|b+b^{\tau}+1=0\}$. Then $\dt\leq 0$, and $\dt=0$ if and only if $p\neq 2$.
For every $g_{a,b}\in U_{\al}$, we have $\val_K(b)\leq 2\val_K(a)+\dt$, and for every $a\in K\m$ there exists
$g_{a,b}\in U_{\al}$ with $\val_K(b)= 2\val_K(a)+\dt$. On the other hand, as it was explained in \cite[Ex. 1.15]{Ti},
for every $g_{0,b}\in U_{2\al}$, we have $\val_K(b)\in 2\B{Z}+\dt+1$.

(c) Using the identification $\C{A}\isom V_{\G,\S}$ corresponding to the standard $\C{O}$-structure of $\G$ and $\S$ (see \re{apart}(a)), we identify the set of affine roots $\Psi(\C{A})$ with the set  \[(\pm\al+\frac{1}{4}(2\B{Z}+\dt))\cup(\pm 2\al+\frac{1}{2}(2\B{Z}+\dt+1)),\] where we divide by an extra  $2$, because our normalization uses valuation $\val_F=\frac{1}{2}\val_K$.

(d) In the notation of (c), for $\psi:=\al+\frac{1}{4}(2n+\dt)$, the subgroup
$U_{\psi}$ consists of all $g_{a,b}\in U_{\al}$ such that $\val_K(b)\geq 2n+\dt$, while for $\psi:=2\al+\frac{1}{2}(2n+\dt+1)$ the subgroup $U_{\psi}$ consists of $g_{0,b}\in U_{\al}$ such that $\val_K(b)\geq 2n+\dt+1$.

(e) Using (d), for every $x\in\C{A}$ and $r\in\B{R}_{\geq 0}$, the subgroup $U_{\al,x,r}$ consists  of $g_{a,b}\in U_{\al}$ such that
$\val_K(b)\geq 4r-4\al(x)$, while the subgroup $U_{2\al,x,r}$ consists of $g_{0,b}\in U_{\al}$ such that $\val_K(b)\geq 2r-4\al(x)$.
In particular, we have $U_{\al,x,r}\cap U_{2\al,x,r}=U_{2\al,x,2r}$.

(f) We claim that an element $g_{a,b}\in U_{\al}$ belongs to $U_{(\al),x,r}$ if and only if we have inequalities $\val_K(a)\geq 2r-2\al(x)-\frac{1}{2}\dt$
and  $\val_K(b)\geq 2r-4\al(x)$.

By definition, $U_{(\al),x,r}$ consists of elements of the form $g_{a,b'+b''}=g_{a,b'}\cdot g_{0,b''}$ such that
$g_{a,b'}\in U_{\al,x,r}$ and $g_{0,b''}\in U_{2\al,x,r}$. In particular, we have
$\val_K(b'')\geq 2r-4\al(x)$, and $\val_K(b')\geq 4r-4\al(x)$ (by (d)), hence $2\val_K(a)\geq \val_{K}(b')-\dt\geq 4r-4\al(x)-\dt$ (by (b)) and
$\val_K(b'+b'')\geq \min\{\val_K(b'),\val_K(b'')\}\geq 2r-4\al(x)$.

Conversely, assume that an element $g_{a,b}\in U_{\al}$ satisfies $\val_K(a)\geq 2r-2\al(x)-\frac{1}{2}\dt$
and  $\val_K(b)\geq 2r-4\al(x)$. Choose $g_{a,b'}\in U_{\al}$ with $\val_K(b')= 2\val_K(a)+\dt$,
and set $b'':=b-b'$. Then $\val_K(b')\geq 4r-4\al(x)$ and $\val_K(b'')\geq \min\{\val_K(b),\val_K(b')\}\geq 2r-4\al(x)$.
Thus $g_{a,b'}\in U_{\al,x,r}$ and $g_{0,b''}\in U_{2\al,x,r}$, hence $g_{a,b'+b''}\in U_{(\al),x,r}$.
\end{Emp}

\begin{Emp} \label{E:levi}
{\bf Levi subgroups.} Let $\L\supseteq\S$ be a Levi subgroup of $\G$, and set $\C{A}_{\L}:=\C{A}_{\L,\S}$.

(a) We have a natural projection $\pr_\L:\C{A}\to \C{A}_{\L}$ of affine spaces, compatible with the projection
$V_{\G,\S}\to V_{\L,\S}$ of vector spaces (see \cite[1.10 and 1.11]{La}).

(b) We have an inclusion $\Phi(\C{A}_{\L})\subseteq\Phi(\C{A})$, and every affine root $\psi\in \Psi(\C{A})$ such that
$\al_{\psi}\in \Phi(\C{A}_{\L})$ induces an affine function $\psi_{\L}$ on $\C{A}_{\L}$, which belongs to $\Psi(\C{A}_{\L})$.
Moreover, the correspondence $\psi\mapsto\psi_{\L}$ induces a bijection between the set of affine roots $\psi\in \Psi(\C{A})$ such that $\al_{\psi}\in\Phi(\C{A}_\L)$ and the set $\Psi(\C{A}_{\L})$.

(c) By definition, for every  $\psi\in \Psi(\C{A})$ such that $\al_{\psi}\in \Phi(\C{A}_{\L})\subseteq\Phi(\CA)$, the affine root subgroup
$U_{\psi}\subseteq U_{\al_{\psi}}$ equals $U_{\psi_{\L}}$.

(d) By (b) and (c), for every $\al\in\Phi(\C{A}_{\L})\subseteq \Phi(\C{A})$, $x\in\C{A}$ and $r\in \B{R}_{\geq 0}$, the affine root subgroup
$U_{\al, x,r}\subseteq U_{\al}\subseteq G$ equals $U_{\al, \pr_{\L}(x),r}\subseteq U_{\al}\subseteq L$.

(e) For every $\al\in\Phi(\C{A})\subseteq X^*(\S)$, let $\S_{\al}$ be the connected component $(\Ker \al)^0$ and set
$\L_{\al}:=\Z_\G(\S_{\al})$. Then $\L_{\al}$ is a Levi subgroup of $\G$ semisimple rank one, thus
$\L^{\sc}_{\al}$ is isomorphic either to $\R_{F'/F}\SL{2}$ or to $\R_{F'/F}\SU{3}$ for some finite separable extension $F'/F$.
\end{Emp}

\begin{Emp} \label{E:weil}
{\bf Weil restriction of scalars.} Let $F'/F$ be a finite separable extension of ramification degree $e$, and set $\G':=\R_{F'/F}\G$. Then we have natural identifications $\G'(F)\cong \G(F')$ and
$\C{X}(\G')\cong\C{X}(\G)$. Moreover, since $\val_{F'}=e\val_{F}$, for every $x\in\C{X}(\G')\cong\C{X}(\G)$ and $r\in\B{R}_{\geq 0}$ the isomorphism $\G'(F)\cong \G(F')$ induces an isomorphism $G'_{x,r}\cong G_{x,er}$.
\end{Emp}

\begin{Emp} \label{E:unr}
{\bf The unramified descent.}
(a) Let $F'/F$, $\G':=\G_{F'}$ and $\S'\subseteq \S$ be as in \rl{f'}. Let $\C{A}'\subseteq\CX(\G')$ be the apartment corresponding to $\S'_{F'}\subseteq\G'$, and set $\Gm':=\Gal(F'/F)$. Then $\C{A}'$ is equipped with an action of $\Gm'$,
and we have a natural identification $\C{A}\isom \C{A}'^{\Gm'}$.

(b) Note that for $\al\in\Phi(\CA)$, the root group  $\U'_{\al}:=(\U_{\al})_{F'}$ equals the product $\prod_{\al'}\U'_{\al'}$, where $\al'$ runs over the union of all $\al'\in \Phi(\CA')$ such that $\al'|_{\C{A}}=\al$ and all $\al'\in \Phi(\CA')_{nd}$ such that $\al'|_{\C{A}}=2\al$ (compare \cite[21.9]{Bo2}).

(c) Moreover, for every $x\in\CA$ and $r\in\B{R}_{\geq 0}$,
the affine root subgroup $U_{\al,x,r}\subseteq U_{\al}$ equals
$U_{\al,x,r}=(U'_{\al,x,r})^{\Gm'}$, where $U'_{\al,x,r}\subseteq U'_{\al}$ is the product
\[
\left(\prod_{\al'\in \Phi(\CA'),\al'|_{\C{A}}=\al} U'_{\al',x,r}\right)\times\left(\prod_{\al'\in \Phi(\CA')_{nd},\al'|_{\C{A}}=2\al} U'_{\al',x,2r}\right),
\]
taken in every order (use, for example, \cite[10.19 and 11.5]{La}).

(d) For every triple $(\al,x,r)$ as in (b),(c) such that $2\al\in \Phi(\CA)$, we have the equality $U_{\al,x,r}\cap U_{2\al,x,r}=U_{2\al,x,2r}$.
Indeed, by (c), it suffices to show that $U'_{\al,x,r}\cap U'_{2\al,x,r}=U'_{2\al,x,2r}$, which reduces to the equality
$U'_{\al',x,r}\cap U'_{2\al',x,r}=U'_{2\al',x,2r}$ for every $\al'\in\Phi(\CA')$ such that  $2\al'\in\Phi(\CA')$. Enlarging $F'$, if necessary, we may assume that $\G'$ splits over a totally ramified extension.  Using \re{levi}(d),(e) and \re{weil}, we reduce to the case $\G=\SU{3}$, in which case the assertion was shown in \re{su3}(e).

(e)  We set $U'_{(\al),x,r}:=U'_{\al,x,r}\cdot U'_{2\al,x,r}\subseteq U'_{\al}$, if  $2\al\in \Phi(\CA)$, and  $U'_{(\al),x,r}:=U'_{\al,x,r}$,  otherwise. We claim that $U_{(\al),x,r}=(U'_{(\al),x,r})^{\Gm'}$. If $2\al\notin \Phi(\CA)$, this follows from (c).
If $2\al\in \Phi(\CA)$, we have to show that $(U'_{\al,x,r}\cdot U'_{2\al,x,r})^{\Gm'}=(U'_{\al,x,r})^{\Gm'}\cdot (U'_{2\al,x,r})^{\Gm'}$. Since $U'_{\al,x,r}\cap U'_{2\al,x,r}=U'_{2\al,x,2r}$ by (d), it suffices to show that $H^1(\Gm',U'_{2\al,x,2r})=0$. Using Shapiro's lemma, the assertion reduces to the vanishing of $H^1(\Gm',\C{O}_{F'})$, which follows from the additive Hilbert 90 theorem.

(f) For every two triples $(\al,x,r)$ and $(\al,y,s)$ as in (b),(c) such that $2\al\in \Phi(\CA)$ we have
$U_{(\al),x,r}\cap U_{(\al),y,s}=(U_{\al,x,r}\cap U_{\al,y,s})\cdot(U_{2\al,x,r}\cap U_{2\al,y,s})$.

Indeed, using (c)-(e) and arguing as in (e), we reduce the assertion to the corresponding equality of the $U'$'s. Then using \re{levi}(d),(e) and \re{weil} we reduce to the case $\G=\SU{3}$, in which case we finish by precisely the same arguments as \re{su3}(f).
\end{Emp}

\begin{Emp} \label{E:lattice}
{\bf Applications}.
(a) Each $\u_{\psi}\subseteq \u_{\al}$ is an $\C{O}$-lattice (see \re{root}(c)). Indeed,
by \re{unr}(c), we reduce to the case when $\G$ is quasi-split and split over a totally ramified extension. Then using \re{levi}(d),(e) and \re{weil}, we reduce to the absolute rank one case, in which case, the assertion follows from formulas  of \re{sl2} and \re{su3}.

(b) For every $x\in \C{A}$, $r\in \B{R}_{\geq 0}$ and $n\in\B{N}$, we have the equality $\varpi^n\g_{x,r}=\g_{x,r+n}$. Again, this can be shown by the same strategy as in (a).

(c) For every $\psi\in\Psi(\C{A})$  there exists a positive integer $n_{\psi}$ such that
the set of $\psi'\in \Psi(\C{A})$ with $\al_{\psi'}=\al_{\psi}$ equals $\psi+\frac{1}{n_{\psi}}\B{Z}$ (see \re{ref}(a)).
Again, we reduce to the absolute rank one case as in (a) and use the explicit formulas from \re{sl2} and \re{su3}.
\end{Emp}

\begin{Lem} \label{L:iwahori}
(a) In the situation of \rp{mp}, the subalgebra $\g_{x,r}$ decomposes as a direct sum
$\g_{x,r}=\mm_r\oplus\prod_{\al\in\Phi(\CA)} \u_{\al,x,r}$.

(b) Assume in addition that either $r>0$ or $x$ lies in a chamber of $[\CX]$.
For every order of $\Phi(\C{A})_{nd}\cup\{0\}$ as in \re{order}(a), the product map
$\prod_{\al\in\Phi(\CA)_{nd}\cup\{0\}} U_{(\al),x,r}\to G_{x,r}$ is bijective.
\end{Lem}

\begin{Emp}
{\bf Remark.} Actually, the map in (b) is bijective for every order of $\Phi(\C{A})_{nd}\cup\{0\}$.
\end{Emp}

\begin{proof}
We show only (b), while the proof of (a) is similar but much easier.

Since $U_{(\al),x,r}\subseteq U_{\al}$ for all $\al$, the injectivity follows from \re{order}(a).
To show the surjectivity, assume first that $\G$ is quasi-split.
In this case, the argument is standard (compare \cite[2.9]{PrR}), and can be carried out as follows.

Let $Y\subseteq G_{x,r}$ be the image of the product map. Since $Y$ is closed, and $\{G_{x,s}\}_{s\geq 0}$
form a basis of open neighbourhoods, it remains to show that $G_{x,r}\subseteq Y\cdot G_{x,s}$ for every $s\geq r$. For this it suffices to show that
$Y\cdot G_{x,s}\subseteq Y\cdot G_{x,s^+}$ for every $s\geq r$. Since $G_{x,s}$ is generated by subgroups $U_{(\al),x,s}$, it remains to show that $Y\cdot U_{(\al),x,s}\subseteq Y\cdot G_{x,s^+}$.

If $s>0$, this follows from the inclusion $(G_{x,r},G_{x,s})\subseteq G_{x,r+s}$ (use \cite[2.4 and 2.7]{PrR}).  If $s=r=0$, and $\al=0$,
this follows from the fact that
$M_r$ normalizes each $U_{(\al),x,r}$. If $\al\neq 0$, then $U_{(\al),x,0}=U_{(\al),x,s}$ for some $s>0$, because $x$ belongs to a chamber,
and the assertion is immediate.

For an arbitrary $\G$, let $F'/F$ and $\G'$ be as in \re{mp}(b). Note that the embedding
$\C{X}(\G)\hra\CX(\G')$ maps chambers into chambers. Set $U'_{(0),x,r}:=M'_{x_{\M},r}$. As it was already shown, the assertion holds for
$G'_{x,r}$ and $M'_{x_{\M},r}$. This implies that the product $\prod_{\al\in\Phi(\CA)_{nd}\cup\{0\}}U'_{(\al),x,r}\to G'_{x,r}$ is bijective. Now the assertion follows from equalities $G_{x,r}=(G'_{x,r})^{\Gm'}$, $M_{r}=(M'_{x_{\M},r})^{\Gm'}$, which were our definitions, and $U_{(\al),x,r}=(U'_{(\al),x,r})^{\Gm'}$ for all $\al\in\Phi(\CA)_{nd}$ (see \re{unr}(d)).
\end{proof}

\begin{Cor} \label{C:iwahori}
Let $(x,r)$ be as in \rl{iwahori}(b), $y\in\C{A}$ and $s\in \B{R}_{\geq 0}$.  Then

(a) For every order of $\Phi(\C{A})_{nd}\cup\{0\}$ as in \re{order}(a), the product map
\[
\prod_{\al\in\Phi(\CA)_{nd}\cup\{0\}} (U_{(\al),x,r}\cap U_{(\al),y,s})\to G_{x,r}\cap G_{y,s}
\]
is bijective.

(b) The subgroup $G_{x,r}\cap G_{y,s}$ is generated by $M_{\max\{r,s\}}$ and the affine root subgroups $U_{\psi}$, where $\psi$ runs over all elements of $\Psi(\C{A})$ such that  $\psi(x)\geq r$ and $\psi(y)\geq s$.
\end{Cor}

\begin{proof}
(a) It follows from \rl{iwahori} that the product map is injective and that every
$g\in G_{x,r}\cap G_{y,s}$ uniquely decomposes as $g=\prod_{\al} g_{\al}$ such that
$g_{\al}\in U_{(\al),x,r}$. It remains to show that $g_{\al}\in U_{(\al),y,s}$ for all $\al$.

If $(y,s)$ also satisfies the assumption of \rl{iwahori}(b), the assertion follows from \rl{iwahori} together with the observation that the product map $\prod_{\al}\U_{\al}\to \G$ is injective. Thus we may assume that $s=0$.

If $r=0$, then, by our assumption, $x$ lies in a chamber of $[\CA]$. Then every $y'\in [x,y)$, close enough to $y$, lies in a chamber
$\si$ such that $y\in\cl(\si)$. Then $g\in G_x\cap G_y\subseteq G_{y'}$ (by \rl{convex}), thus $g\in G_x\cap G_{y'}$. Thus, by the previous case, $g_{\al}\in U_{(\al),y',0}\subseteq U_{(\al),y,0}$.

Finally, if $r>0$, then there exists a point $x'\in\C{A}$, lying in a chamber of $[\C{A}]$ such that
$G_{x,r}\subseteq G_{x'}$. Thus, $g\in G_{x'}\cap G_{y}$, hence $g_{\al}\in U_{(\al),y,0}$ by the
$r=0$ case.

(b) The assertion (b) follows from (a) and \re{unr}(f).
\end{proof}

\begin{Emp} \label{E:pfmp}
\begin{proof}[Proof of \rp{mp}]
\rl{iwahori} implies all the cases, except the case of $G_x$, which is not Iwahori.
To show the remaining case (which is not used in this work), note that $G_x$ is generated by its Iwahori subgroups $G_y$, where $y$ lies in a chamber $\si\subseteq\CA$ such that $x\in\cl(\si)$. Since  each $G_y$ is generated by $T_0$ and $U_{\psi}$ with $\psi(y)\geq 0$ by \rl{iwahori}(b), and inequality $\psi(y)\geq 0$ implies $\psi(x)\geq 0$, the assertion for $G_x$ follows as well.
\end{proof}
\end{Emp}

\begin{Emp} \label{E:pfconvex}
\begin{proof}[Completion of the proof of \rl{convex}]
As indicated in \re{bpfconvex}, it remains to show that for every $x,y\in\CX$ and $z\in[x,y]$,   we have
$G_{x,r}\cap G_{y,r}\subseteq G_{z,r}$ for $r>0$ and $\g_{x,r}\cap \g_{y,r}\subseteq \g_{z,r}$ for $r\geq 0$.
Choose an apartment $\CA$ of $\CX$ such that $x,y\in\CA$.

By \rco{iwahori}(b), to show that $G_{x,r}\cap G_{y,r}\subseteq G_{z,r}$ for $r>0$ it suffices to show that for every $\psi\in\Psi(\CA)$
such that $\al(x)\geq r$ and $\al(y)\geq r$ we have $\al(z)\geq r$. But this follows from the assumption $z\in[x,y]$.
The proof of the inclusion $\g_{x,r}\cap \g_{y,r}\subseteq \g_{z,r}$ is similar, but easier.
\end{proof}
\end{Emp}

\begin{Emp} \label{E:pfadler}
\begin{proof}[Proof of \rl{adler}]
We show the assertion for $G_{x,r}$, while the assertion for $\g_{x,r}$ is similar but easier.
For $r=0$, the assertion follows from the \re{parah}(d).

Assume now that $r>0$.
Enlarging $F^{\flat}$ if necessary, we can assume that $\G^{\flat}$ is split. Next, replacing $F$ by a finite unramified extension and using \re{remmp}(a), we can assume that $\G$ is quasi-split. Then, using \rl{iwahori}, it remains to show the corresponding equality for tori $T_r=T^0\cap T^{\flat}_{re}$, which was our definition, and a similar equality for each affine root subgroup $U_{(\al),x,r}$.

Finally, using \re{levi}(d),(e) and \re{weil}, we reduce to the case of $\SL{2}$ and $\SU{3}$, which follow from
formulas in \re{sl2} and \re{su3}(f), respectively.
\end{proof}
\end{Emp}

\begin{Emp}
{\bf Remark.}
The formula of \re{su3}(f) also implies that the conclusion of \rl{adler} is false, if $\G$ is
$\SU{3}$, split over a wildly ramified quadratic extension. 
\end{Emp}

\section{Congruence subsets} \label{S:cong}

\begin{Emp}
{\bf Notation.} For every $r\in \B{R}_{\geq 0}$,
we set  $G_r:=\cup_{x\in \CX}G_{x,r}\subseteq G$ and $\g_r:=\cup_{x\in \CX}\g_{x,r}\subseteq \g$.
By construction, both $G_r\subseteq G$ and  $\g_r\subseteq\g$ are open and $\Ad G$-invariant. Moreover, we have
$G_{r^+}=\cup_{s>r}G_s\subseteq G_r$ and $\g_{r^+}=\cup_{s>r}\g_s\subseteq\g_r$.
\end{Emp}

\begin{Emp} \label{E:disc}
{\bf Remark.} (a) The set of $r\in\B{R}_{\geq 0}$ such that $G_{r^+}\neq G_r$ (resp. $\g_{r^+}\neq\g_r$) is discrete.
For example, this follows from the fact that any such $r$ is optimal in the sense of \cite[2.3]{ADB}. Alternatively, this can be seen as follows.

Choose any $r$ such that $G_{r^+}\neq G_r$, and choose a chamber $\si\in [\CX]$. Since all chambers are $G$-conjugate, there exists
$x\in \cl(\si)$ such that $G_{x,r}\neq G_{x,r^+}$. Choose $k\in\B{Z}$ such that $r\in(k,k+1]$.
It thus suffices to show that the set of subgroups  $\{G_{x,s}\}_{x\in\cl(\si),s\in(k,k+1]}$ is finite.

Choose an apartment $\CA\subseteq\CX$ containing $\si$, and fix $x\in\cl(\si)$ and $s\in(k,k+1]$.
Then the set $\{\psi\in \Psi(\CA)\,|\,\psi(x)\geq s\}$ contains the set $\{\psi\in \Psi(\CA)\,|\,\psi(\si)> k+1\}$ and is contained in
the set $\{\psi\in \Psi(\CA)\,|\,\psi(\si)> k\}$. This implies the assertion.

(b) It can be shown that every $r$ from (a) is rational. But even without this fact it follows
from (a) that for every $r\in\B{R}_{\geq 0}$, there exist $r',r''\in\B{Q}_{\geq 0}$ such that
$G_r=G_{r'} $ and $G_{r^+}=G_{r''}$, and similarly for $\g$. Thus \rl{clos} follows from
the following assertion.
\end{Emp}

\begin{Lem} \label{L:closed}
For every $r\in\B{R}_{\geq 0}$, the subsets $G_{r}\subseteq G$, $\g_{r}\subseteq \g$ and $\g^*_{-r}\subseteq\g^*$ are open, closed and stable.
\end{Lem}

\begin{Emp}
{\bf Remark.} Under some mild restriction on the residual characteristic of $F$ one can show a more precise result (with a simpler proof) asserting that $G_{r}$ (resp. $\g_{r}$, resp. $\g^*_{-r}$) is equal to the full preimage of a certain open and compact subset of the corresponding Chevalley space.
\end{Emp}

\begin{proof}
First we show that $G_0\subseteq G$ is closed. By \re{parah}, the subgroup $G^0\subseteq G$ is closed, and $G_0=\cup_{x\in\CX}\Stab_{G^0}G(x)$.
Then, by the Bruhat--Tits fixed point theorem, $G_0$ coincides with the set of
all compact elements of $G^0$. But the set of all compact elements of $G$ is closed.
Indeed, choose a faithful representation $\rho:\G\hra\GL_n$, and notice that $g\in G$ is compact
if and only if  $\det\rho(g)\in\C{O}\m$ and the characteristic polynomial of $\rho(g)$ has coefficients in $\C{O}$.

Next we show that $G_r\subseteq G$ is closed for $r>0$. Since $G_0=\cup_{x\in\CX} G_{x}\subseteq G$ is closed,
and each $G_x$ is open and compact, it remains to show that for every $x\in\CX$, the intersection $G_x\cap G_r$ is compact.
By \re{disc}(b), we may assume that $r\in\B{Q}$, hence $r\in \frac{1}{m}\B{Z}_{\geq 0}$ for some $m\in\B{N}$.
As in \re{pflie}, for every $\Si\in\T_m$, we set $G_{\Si,r}:=\cup_{\si\in \Si}G_{\si,r}$. Then each
$G_x\cap G_{\Si,r}$ is compact, and it suffice to show that
$G_x\cap G_r= G_x\cap G_{\Si,r}$ for every $\Si\supseteq\Upsilon_{x,r}$. Equivalently, it suffices to show
 the equality of functions $1_{G_x}\cdot 1_{G_{\Si,r}}=1_{G_x}\cdot 1_{G_{\Si',r}}$ for every
$\Si',\Si\in\T_m$ such that $\Upsilon_{x,r}\subseteq\Si'\subseteq\Si$.

As in \rl{ep}, we deduce from \rl{convex} that for every $\Si\in \T_m$ we have
$1_{G_{\Si,r}}=\sum_{\si\in\Si}(-1)^{\dim\si}1_{G_{\si,r}}$.
Thus we have to show that for every $\Si'\subseteq\Si$ as above, we have
$\sum_{\si\in\Si\sm\Si'}(-1)^{\dim\si}(1_{G_x}\cdot 1_{G_{\si,r}})=0$. Arguing as in \rp{stab}(a),
it remains to show that for every $\si,\si'\in [\CX_m]$ with $\si'\preceq\si$ and $\si\in\Gm_r(\si',x)$
we have the equality $1_{G_x}\cdot 1_{G_{\si,r}}=1_{G_x}\cdot 1_{G_{\si',r}}$. Equivalently, we have to show that
$G_x\cap G_{\si,r}=G_x\cap G_{\si',r}$, that is, $G_x\cap G_{\si',r}\subseteq G_{\si,r}$.

Choose an apartment $\CA\subseteq\CX$, containing  $\si,x$. By \rco{iwahori}(b), the intersection $G_x\cap G_{\si',r}$
is generated by $M_r$ and the affine root subgroups $U_{\psi}$, where $\psi\in \Psi(\CA)$ satisfies $\psi(x)\geq 0$
and $\psi(\si')\geq r$. Thus we have to show that for every $\psi\in \Psi(\CA)$ such that $\psi(x)\geq 0$ and
$\psi(\si')\geq r$ we have $\psi(\si)\geq r$. Replacing $\psi$ with $r-\psi$ it suffices to show that for every
$\psi\in \Psi_m(\CA)$ with $\psi(x)\leq r$ and $\psi(\si')\leq 0$, we have $\psi(\si)\leq 0$. But this is precisely
the assumption $\si\in\Gm_r(\si',x)$.

This shows that every $G_r$ is closed. To show that $G_r$ is stable, we need to show that for every $\G(\ov{F})$-conjugate
$g,g'\in G^{\sr}$ such that $g\in G_r$, we have $g'\in G_r$. In other words, we have to show that the subset
$\CX(g',r)\subseteq\CX$ consisting of all $x\in\CX$ such that $g'\in G_{x,r}$, is non-empty.

Since $g$ and $g'$ are $\G(\ov{F})$-conjugate, and
$F^{\nr}$ is of cohomological dimension one, we conclude that  $g,g'$ are $\G(F^{\nr})$-conjugate,
thus $\G(F^{\flat})$-conjugate for some finite unramified extension $F^{\flat}/F$. Set $\G^{\flat}:=\G_{F^{\flat}}$,
$\CX^{\flat}:=\CX(\G^{\flat})$ and $\Gm^{\flat}:=\Gal(F^{\flat}/F)$.
Then $g\in G_r\subseteq G^{\flat}_r$, hence $g'\in G\cap G^{\flat}_r$, because $G^{\flat}_r$ is $\Ad G^{\flat}$-invariant.
Thus the subset $\CX^{\flat}(g',r)\subseteq\CX^{\flat}$ is non-empty. On the other hand, $\CX^{\flat}(g',r)$ is
$\Gm^{\flat}$-invariant, because $g'\in G$, and convex, by \rl{convex}. Thus, by the Bruhat--Tits fixed point theorem,
the set of fixed points $\CX^{\flat}(g',r)^{\Gm^{\flat}}$ is non-empty.
Since $\CX^{\flat}(g',r)^{\Gm^{\flat}}$ equals $\CX^{\flat}(g',r)\cap\CX=\CX(g',r)$ (by \re{bc} and \re{remmp}(a)),
we are done.

The proof for $\g_{r}$ is similar. Namely, for every
$x\in\CX$, we have $\g=\cup_n \varpi^{-n}\g_x$. Thus to show that  $\g_{r}\subseteq \g$ is closed, it remains to show that every
$\g_{r}\cap \varpi^{-n}\g_x$ is compact. Since $\varpi^n\g_r=\g_{r+n}$ (see \re{lattice}(b)), it remains to show
that the intersection $\g_{r+n}\cap \g_x$ is compact. This can be shown as in the group case.

Finally, the prove the result for $\g^*_{-r}$ we can either mimic the proof for $\g_{r}$, using the decomposition for $\g^*_{x,-r}$,
obtained from \rl{iwahori}(a) by duality, or to deduce it from a Lie algebra version of \rp{stab}(a) by the Fourier transform.
\end{proof}

\begin{Emp} \label{E:pfisog}
\begin{proof}[Proof of \rl{isog}]
It suffices to show that $\pi$ induces bijections $\pi_x:G'_{x,r^+}\isom G_{x,r^+}$ and $\pi_{x,y}:G'_{x,r^+}\cap G'_{y,r^+}\isom G_{x,r^+}\cap G_{y,r^+}$ for $x,y\in\CX(\G)=\CX(\G')$. Indeed, the surjectivity of $G'_{r^+}\isom G_{r^+}$ follows from the surjectivity of the $\pi_x$'s, while injectivity follows from the surjectivity of the $\pi_{x,y}$'s and the injectivity of the $\pi_x$'s.

Replacing $F$ by $F'$ as in \re{unr}, we may assume that $\G$ and $\G'$ are quasi-split over $F$. Choose an apartment $\CA\ni x,y$, corresponding to a maximal split torus $\S\subseteq\G$, and set $\TT:=\Z_\G(\S)$, and $\TT':=\pi^{-1}(\TT)\subseteq\G'$. Then
$\TT\subseteq \G$ is a maximal torus, and we have decompositions
$G_{x,r^+}=T_{r^+}\times \prod_{\al}U_{(\al),x,r^+}$ (by \rl{iwahori}) and $G_{x,r^+}\cap G_{y,r^+}=T_{r^+}\times \prod_{\al}(U_{(\al),x,r^+}\cap U_{(\al),y,r^+})$
(by \rco{iwahori}), and similarly for $\G'$.

Since $\pi$ induces isomorphisms between the $U_{\al}$'s,
it remains to show that the induced map $T'_{r^+}\to T_{r^+}$ is an isomorphism. If $\TT$ and $\TT'$ are split, the assertion is easy.
Namely, $\pi$ induces a morphism of $\B{F}_p$-vector spaces $\pi_n:T'_n/T'_{n+1}\to T_n/T_{n+1}$ for every $n>0$.
Hence each $\pi_n$ is an isomorphism, because the degree of $\pi$ is prime to $p$. Therefore $T'_{r^+}\to T_{r^+}$ is an isomorphism as well.

In general, let $F^{\flat}$ be the splitting field of $\TT$ (and $\TT'$), and let $e$ be the ramification degree of $F^{\flat}/F$. Set $r^{\flat}:=er$, and  $\Gm^{\flat}:=\Gal(F^{\flat}/F)$. Then $T_{r^+}=\Ker w_{\TT}\cap \TT(F^{\flat})^{\Gm^{\flat}}_{(r^{\flat})^+}$, where $w_\TT$ is the Kottwitz homomorphism $\TT(F^{\nr})\to X_*(\TT)_{\Gm_{\nr}}$ (see \re{filtor}(a)), and similarly for $\TT'$.  By the split case, $\pi$ induces an isomorphism $\TT'(F^{\flat})^{\Gm^{\flat}}_{(r^{\flat})^+}\isom\TT(F^{\flat})^{\Gm^{\flat}}_{(r^{\flat})^+}$ of pro-$p$-groups. By the functoriality of the Kottwitz homomorphism, it remains to check that every element in the kernel of the homomorphism $X_*(\TT')_{\Gm_{\nr}}\to X_*(\TT)_{\Gm_{\nr}}$ is torsion of prime to $p$ order. Since this kernel is killed by $\deg\pi$, the proof is complete.
\end{proof}
\end{Emp}

\section{Quasi-logarithms} \label{S:qlog}

\begin{Emp} \label{E:qlog}
{\bf Quasi-logarithms.} Let $\G$ be a reductive group over a field $F$.

(a) Following \cite[1.8]{KV}, we call an  $\Ad \G$-equivariant morphism of algebraic varieties $\C{L}:\G\to \CG$ a {\em quasi-logarithm}, if
$\C{L}(1)=0$, and the induced map on tangent spaces $d\C{L}_1:\g=T_1(G)\to T_0(\g)=\g$ is the identity map.

(b) Let $F^{\flat}/F$ be a field extension. Then a quasi-logarithm $\C{L}:\G\to\CG$ induces a quasi-logarithm
$\C{L}_{F^{\flat}}:\G_{F^{\flat}}\to \CG_{F^{\flat}}$. Conversely, a quasi-logarithm $\C{L}^{\flat}:\G_{F^{\flat}}\to \CG_{F^{\flat}}$ induces a
quasi-logarithm $\R_{F^{\flat}/F}(\C{L}^{\flat}):\R_{F^{\flat}/F}(\G_{F^{\flat}})\to \R_{F^{\flat}/F}(\CG_{F^{\flat}})=\CG\otimes_F F^{\flat}$.

(c) Since $\C{L}$ is $\Ad\G$-equivariant, it induces a morphism $[\C{L}]:\c_\G\to \c_{\CG}$ of the corresponding Chevalley spaces
(compare \cite[5.2]{KV2}).
\end{Emp}

\begin{Emp} \label{E:qlogo}
{\bf Quasi-logarithms defined over $\C{O}$.} Let $F$ be a local non-archimedean field of residual characteristic $p$.

(a) Assume that $\G$ is split over $F$. Then the Chevalley spaces $\c_\G$ and $\c_{\CG}$  have natural structures over
$\C{O}$. In this case, we say that a quasi-logarithm $\C{L}:\G\to\CG$ is {\em defined over $\C{O}$},
if the corresponding map  $[\C{L}]$ is defined over $\C{O}$ (compare  \cite[5.2]{KV2}). Note that by
\cite[Lem 5.2.1]{KV2} this notion is equivalent to the corresponding notion of \cite[1.8.8]{KV}.

(b) For an arbitrary $\G$, we say that $\C{L}:\G\to\CG$ is defined over $\C{O}$, if
$\C{L}_{F^{\flat}}$ is defined over $\C{O}_{F^{\flat}}$ for some or, equivalently, every splitting field $F^{\flat}$ of $\G$.

(c) Let $F^{\flat}/F$ be a finite Galois extension, and let $\C{L}^{\flat}:\G_{F^{\flat}}\to \CG_{F^{\flat}}$ be a quasi-logarithm defined
over $\C{O}_{F^{\flat}}$. Then the quasi-logarithm $\R_{F^{\flat}/F}(\C{L}^{\flat})$ (see \re{qlog}(b)) is also defined over $\C{O}$.

(d) In the situation of (c), assume that $[F^{\flat}:F]$ is prime to $p$. Then the composition
\[
\C{L}:\G\hra \R_{F^{\flat}/F}\G_{F^{\flat}}\overset{\R_{F^{\flat}/F}(\C{L}^{\flat})}{\lra}\CG\otimes_F F^{\flat}\overset{\frac{1}{[F^{\flat}:F]}\Tr_{F^{\flat}/F}}{\lra}\CG
\]
is a quasi-logarithm defined over $\C{O}$.
\end{Emp}

\begin{Lem} \label{L:qlogo}
Assume that $\G$ is semisimple and simply connected and $p$ is very good for $\G$ (see \re{verygood}).
Then $\G$ admits a quasi-logarithm defined over $\C{O}$.
\end{Lem}

\begin{proof}
(compare \cite[Lem 1.8.12]{KV}). Assume that $\G=\prod_i \R_{F_i/F}\H_i$ as in \re{verygood}. By \re{qlogo}(c), we can replace $\G$ by $\H_i$, thus assuming that $\G$ is absolutely simple. Using \cite[Lem 1.8.9]{KV}, we can replace $\G$ by its quasi-split inner form.
Since $p$ is good, $\G$ splits over a tamely ramified extension. Hence,
using \re{qlogo}(d), we may extend scalars to the splitting field of $\G$, thus assuming that $\G$ is split. In this case,
the assertion was shown in \cite[Lem 1.8.12]{KV}, using the fact that $\G$ has a faithful representation, whose Killing form is
non-degenerate over $\C{O}$.  Namely, one uses the standard representation, if $\G$ is classical, and the adjoint representation, if
$\G$ is exceptional.
\end{proof}

\begin{Lem} \label{L:qlog}
Let $\G$ be semisimple and simply connected, $p\neq 2$, and let  $\C{L}:\G\to \CG$ be a quasi-logarithm defined over $\C{O}$. Then for every
$x\in \CX$ and $r\in\B{R}_{\geq 0}$, $\C{L}$ induces analytic isomorphisms $\C{L}_{r}:G_{r^+}\isom\g_{r^+}$ and $\C{L}_{x,r}:G_{x,r^+}\isom\g_{x,r^+}$.
\end{Lem}

\begin{proof}
Assume first that $\G$ is split.
The assertion for $r=0$ was shown in \cite[Prop 1.8.16]{KV}. Next we show that $\C{L}$ induces an analytic isomorphism $\C{L}_{x,r}:G_{x,r^+}\to \g_{x,r^+}$ when
$x\in \CX$ is a hyperspecial vertex and $r=n\in\B{Z}$. In this case, $G_{x,r^+}=G_{x,n+1}$ and
$\g_{x,r^+}=\g_{x,n+1}$, so we have to show that $\C{L}$ induces an analytic isomorphism $G_{x,n+1}\isom\g_{x,n+1}$.
This is easy and it was shown in the course of the proof of \cite[Prop 1.8.16]{KV}. We are going to deduce the general case
from the particular case shown above.

Let $F^{\flat}/F$ be a finite Galois extension of ramification degree $e$, and set $r^{\flat}:=er$, $\Gm^{\flat}:=\Gal(F^{\flat}/F)$ and  $\G^{\flat}:=\G_{F^{\flat}}$. Then $\C{L}$ induces a quasi-logarithm $\C{L}^{\flat}:=\C{L}_{F^{\flat}}: \G^{\flat}\to \CG^{\flat}$, which is $\Gm^{\flat}$-equivariant and defined over $\C{O}_{F^{\flat}}$.
Moreover, since $\G$ is semisimple and
simply connected, we have $G^0=G$ (see \re{parah}). Since $p\neq 2$, we have  $G_{x,r^+}=(G^{\flat}_{x,(r^{\flat})^+})^{\Gm^{\flat}}$ and  $\g_{x,r^+}=(\g^{\flat}_{x,(r^{\flat})^+})^{\Gm^{\flat}}$
(by \rl{adler}).

Note that the assertion for $\C{L}^{\flat}$ and $r^{\flat}$ implies that for $\C{L}$ and $r$. Indeed, if $\C{L}^{\flat}$ induces an isomorphism
$\C{L}^{\flat}_{x,r^{\flat}}$, then it is automatically $\Gm^{\flat}$-equivariant, thus induces an isomorphism
$\C{L}_{x,r}:=(\C{L}^{\flat}_{x,r^{\flat}})^{\Gm^{\flat}}$ of Galois invariants. Therefore $\C{L}$ induces a morphism
$\C{L}_{r}:G_{r^+}\to \g_{r^+}$, which is surjective, because each  $\C{L}_{x,r}$ is surjective, and injective, because  $\C{L}^{\flat}_{r^{\flat}}$ is injective. Thus we can replace $F$ by $F^{\flat}$, $\G$ by $\G^{\flat}$, and $r$ by $r^{\flat}$.

Now the assertion is easy. Indeed, choosing $F^{\flat}$ to be a splitting field of $\G$, we can assume that $\G$ is split.
Since $\C{L}_{0}$ is injective, it is enough to show that $\C{L}$ induces an isomorphism $\C{L}_{x,r}$.
Observe that both $G_{x,r^+}$ and $\g_{x,r^+}$ do not change if we replace pair
$(x,r)$ by a close pair $(x',r')$. Thus we may assume
that $r\in\frac{1}{m}\B{Z}_{\geq 0}$ and $x$ is a hyperspecial vertex of $[\CX_m(\G)]$ for some $m$.

Choose a finite extension $F^{\flat}$ of $F$ of ramification degree $m$. Then $r^{\flat}=mr\in\B{N}$,  and
$x$ is a hyperspecial vertex of $[\CX_m(\G)]\subseteq [\CX(\G^{\flat})]$. Hence the assertion for
$\C{L}^{\flat}_{x,r^{\flat}}$, shown in the first paragraph of the proof, implies the assertion for $\C{L}_{x,r}$.
\end{proof}

\end{document}